\include{birkart.sty}

\documentclass{amsart}
\usepackage{amsmath}
\usepackage{amssymb}
\usepackage{amsfonts}
\usepackage{graphicx}
\usepackage{texdraw}
\usepackage{graphpap}
\usepackage{wasysym}
\usepackage{pxfonts}
\usepackage[OT2,T1]{fontenc}
\DeclareSymbolFont{cyrletters}{OT2}{wncyr}{m}{n}
\DeclareMathSymbol{\berd}{\beta}{cyrletters}{"42}

\input txdtools

 \newtheorem{thm}{Theorem}[section]
 \newtheorem{cor}[thm]{Corollary}
 \newtheorem{lem}[thm]{Lemma}
 \newtheorem{prop}[thm]{Proposition}
 \theoremstyle{definition}
 \newtheorem{defn}[thm]{Definition}

 \theoremstyle{remark}
 \newtheorem{rem}[thm]{Remark}

\numberwithin{equation}{section}
\numberwithin{figure}{section}

\renewcommand{\d}{{\partial}}
\newcommand{\dbar}{\overline{\partial}}
\newcommand{\e}{e}
\newcommand{\setS}{{\mathcal S}}

\newcommand{\C}{{\mathbb C}}

\newcommand{\T}{{\mathbb T}}
\newcommand{\R}{{\mathbb R}}

\newcommand{\1}{{\mathbf 1}}

\newcommand{\wt}{\widetilde}
\newcommand{\wh}{\widehat}

\newcommand\coity{{\mathcal C \sb 0 \sp \infty}}

\newcommand{\calC}{{\mathcal C}}

\newcommand{\calB}{{\mathcal B}}

\newcommand{\calO}{{\mathcal O}}
\newcommand{\dA}{{\diff A}}

\newcommand{\calE}{{\mathcal E}}
\newcommand{\calW}{{\mathcal W}}

\newcommand{\calP}{{\mathcal P}}

\newcommand{\calS}{{\mathcal S}}

\newcommand{\re}{\operatorname{Re}}
\newcommand{\im}{\operatorname{Im}}

\newcommand{\diff}{{\mathrm d}}
\newcommand{\imag}{{\mathrm i}}

\newcommand{\Cov}{\operatorname{Cov}}

\newcommand{\dist}{\operatorname{dist}}
\newcommand{\supp}{\operatorname{supp}}

\newcommand{\SH}{{\rm SH}}
\newcommand{\eps}{\varepsilon}

\newcommand{\diag}{\operatorname{diag}}
\newcommand{\vt}{\vartheta}
\newcommand{\vf}{\varphi}

\newcommand{\trace}{\operatorname{trace}}
\newcommand{\fluct}{\operatorname{fluct}}
\newcommand{\Vol}{\operatorname{Vol}}
\newcommand{\Area}{\operatorname{Area}}

\def\lpar{\left (}
\def\rpar{\right )}
\def\labs{\left |}
\def\rabs{\right |}
\def\babs#1{\labs {#1} \rabs}

\begin{document}
\title[fluctuations]
{Fluctuations of eigenvalues of random normal matrices}
\begin{abstract} In this note, we consider a fairly general potential in the plane and the corresponding  Boltzmann--Gibbs distribution of eigenvalues of random normal matrices. As the order of the matrices
tends to infinity, the eigenvalues condensate on a certain
compact subset of the plane -- the "droplet''.
We give two proofs for the Gaussian field convergence of fluctuations of linear statistics of eigenvalues of random normal matrices in the interior of the droplet. We also discuss various ramifications of this result.
\end{abstract}

%----------Author 1
\author[Ameur]
{Yacin Ameur}
\address{Yacin Ameur\\ Department of Mathematics\\
Uppsala University\\
Box 480\\
751 06 Uppsala\\
Sweden}

\email{yacin.ameur@gmail.com}

\thanks{Research supported by the G\"oran Gustafsson Foundation. The third author is supported
by N.S.F. Grant No. 0201893.}

%----------Author 2
\author[Hedenmalm]
{H\aa{}kan Hedenmalm}

\address{Hedenmalm: Department of Mathematics\\
The Royal Institute of Technology\\
S -- 100 44 Stockholm\\
Sweden}

\email{haakanh@math.kth.se}

%----------Author 3

\author[Makarov]
{Nikolai Makarov}

\address{Makarov: Mathematics\\
California Institute of Technology\\
Pasadena, CA 91125\\
USA}

\email{makarov@caltech.edu}

%%% ----------------------------------------------------------------------
\maketitle
%%% ----------------------------------------------------------------------

\addtolength{\textheight}{2.2cm}

\section{Notation, preliminaries and the main result}\label{main1}
\subsection*{Random normal matrix ensembles} Let a \textit{weight function} (or \textit{potential}) $Q:\C\to\R$
be fixed. We assume throughout that $Q$ is $\calC^\infty$ on
$\C$ (sometimes excepting a finite set where the value may be $+\infty$) and that there are positive numbers $C$ and $\rho$ such that
\begin{equation}\label{gro}Q(z)\ge \rho\log\babs{z}^2,\quad \babs{z}\ge
C.\end{equation}

Let ${\mathfrak N}_n$ be the space of all normal $n\times n$ matrices $M$ (i.e., such that $M^*M=MM^*$) with metric induced
from the standard metric on the space $\C^{n^2}$ of
all $n\times n$ matrices. Write $M=UDU^*$ where
$U$ unitary, i.e. of class ${\mathfrak U}_n$, and $D=\diag(\lambda_i)\in \C^n$.

It is well-known \cite{CZ}, \cite{EF} that the Riemannian volume form
on ${\mathfrak N}_n$ is given by
$\diff M_n:= \diff U_n ~\babs{V_n(\lambda_1,\ldots,\lambda_n)}^2~\diff^2\lambda_1\cdots
\diff^2\lambda_n,$
where $\diff U_n$ is the normalized ${\mathfrak U}_n$-invariant measure
on ${\mathfrak U}_n/\T$,
and $V_n$ is the Vandermonde determinant
\begin{equation*}V_n(\lambda_1,\ldots,\lambda_n)=
\prod_{j<k}(\lambda_j-\lambda_k).
\end{equation*}

We introduce another parameter $m\ge 1$ and consider the probability measure (on ${\mathfrak N}_n$)
\begin{equation*}\diff P_{m,n}(M)=\frac 1 {C_{m,n}}~\e^{-m\trace
Q(M)}~\diff M_n,\end{equation*}
where $C_{m,n}$ is the normalizing constant making the total mass equal to one.

In random matrix theory, it is common to study fluctuation properties of the spectrum. In the present case this means that one disregards the unitary part of
$P_{m,n}$ and passes to the
following probability measure on $\C^n$ (the \textit{density of states}),
\begin{equation}\label{vander}\diff \Pi_{m,n}(\lambda_1,\ldots,\lambda_n)=
\frac 1 {Z_{m,n}}~\babs{V_n(\lambda_1,\ldots,\lambda_n)}^2 ~\e^{-m\sum_{j=1}^n
Q(\lambda_j)}~\dA_n(\lambda_1,\ldots,\lambda_n),
\end{equation}
 Here the \textit{partition function} $Z_{m,n}$ is given by
\begin{equation}\label{zmno}Z_{m,n}=\int_{\C^n}
\babs{V_n(\lambda_1,\ldots,\lambda_n)}^2
~\e^{-m\sum_{j=1}^n
Q(\lambda_j)}~\dA_n(\lambda_1,\ldots,\lambda_n),\end{equation} where we put
$\dA_n(\lambda_1,\ldots,\lambda_n)=\dA(\lambda_1)\cdots\dA(\lambda_n)$;
$\dA(z)=\diff^2 z/\pi$ is the suitably normalized area
measure in the plane. (The integral \eqref{zmno} converges when $m/n>\rho^{-1}$; we always assume that this is the case.)

Now fix a number $\tau$ such that
\begin{equation*}0<\tau<\rho.\end{equation*}

We can think of the eigenvalues $(\lambda_i)_1^n$ as a system of point charges (electrons) confined to a plane, under the influence
of the external magnetic potential $Q$ \cite{Z}.
In the limit when $m\to\infty$, $n/m\to \tau$, the growth condition \eqref{gro} on $Q$ is sufficient to force the point charges to
condensate on a certain finite portion of the plane, called the "droplet'', the details of which depends on $Q$ and $\tau$. Thus the system of electrons, the \textit{Coulomb gas}, lives in the vicinity of the droplet. Inside the droplet the repulsive behaviour of the point charges takes overhand and causes them to be very evenly spread out there.

\subsection*{The droplet} We review some elements from
weighted potential theory.
Let $\Delta$ denote the normalized Laplacian,
$\Delta=\d\dbar$ where $\d=\frac 1 2(\d_x-\imag\d_y)$ and
$\dbar=\frac 1 2(\d_x+\imag\d_y)$.
Write
\begin{equation}\label{xdef}X=\{\Delta Q>0\}.\end{equation}
Let $\SH_\tau$ denote the set of subharmonic functions $f:\C\to \R$
such that $f(z)\le \tau\log_+\babs{z}^2+\calO(1)$ as $z\to
\infty$. The \textit{equilibrium potential} $\widehat{Q}_\tau$ is
defined as the envelope
\begin{equation*}\widehat{Q}_\tau(z)=\sup\left\{f(z);~ f\in \SH_\tau,\,
f\le Q\quad\text{on}\quad \C\right\}.\end{equation*} The \textit{droplet}
associated with the number $\tau$ is the set
\begin{equation}\label{stdef}\setS_\tau=\left\{Q=
\widehat{Q}_\tau\right\}.\end{equation}
Our assumptions then imply that $\widehat{Q}_\tau\in \SH_\tau$,
$\widehat{Q}_\tau\in \calC^{1,1}(\C)$, $\setS_\tau$ is a compact set
and $\widehat{Q}_\tau$ is harmonic outside $\setS_\tau$. See e.g.
\cite{ST} or \cite{HM}.
In
particular, since $z\mapsto \tau\log_+(\babs{z}^2/C)-C$ is a
subharmonic minorant of $Q$ for large enough $C$, it yields that
\begin{equation}\label{beq}\widehat{Q}_\tau(z)=\tau\log_+\babs{z}^2+\calO(1)\quad
\text{on}\quad \C.
\end{equation}

Let $\calP$ be the convex set of all compactly supported Borel
probability measures on $\C$. The \textit{energy functional}
corresponding to $\tau$ is given by
\begin{equation*}I_\tau(\sigma)=
\int_{\C^2}\lpar\log\frac 1 {\babs{z-w}}+\frac {Q(z)+Q(w)}
{2\tau}\rpar \diff\sigma(z)\diff\sigma(w),\quad
\sigma\in\calP.\end{equation*} There then exists a unique
\textit{weighted equilibrium measure} $\sigma_\tau\in \calP$ which
minimizes the energy $I_\tau(\sigma)$ over all $\sigma\in\calP$.
Explicitly, this measure is given by
\begin{equation*}\diff\sigma_\tau(z)=
\tau^{-1}\Delta\widehat{Q}_\tau(z)~\dA(z)= \tau^{-1}\Delta Q(z)\1_{\setS_\tau\cap X}(z)~\dA(z).
\end{equation*}
Cf. \cite{ST}, \cite{HM}.

The problem of determining the details of the droplet are known under the names "Laplacian growth'' or "quadrature domains''. When
$Q$ is real-analytic in a nieghbourhood of the droplet, the boundary of the droplet is a finite union of analytic arcs with at most a finite number of singularities which can be either
cusps pointing outwards from the droplet, or double-points. (\footnote{We will
discuss this result in detail elsewhere})
On the other hand, if $Q$ is just $\calC^\infty$-smooth, the boundary
will in general be quite complicated.

\subsection*{The correlation kernel}
We state a couple of well-known facts concerning the measure
$\Pi_{m,n}$, \eqref{vander}. For positive integers $n$ with $n<m\rho$ we let $H_{m,n}$ be the
space of analytic polynomials of degree at most $n-1$ with inner
product
$\langle f,g\rangle_{mQ}=\int_\C f(z)\overline{g(z)}~\e^{-mQ(z)}~\dA(z)$.
We denote by $K_{m,n}$ the reproducing kernel for $H_{m,n}$, i.e.,
\begin{equation*}K_{m,n}(z,w)=\sum_{j=1}^n \phi_j(z)\overline{\phi_j(w)},
\end{equation*}
where $\{\phi_j\}_{j=1}^n$ is an orthonormal basis for $H_{m,n}$.

It is well-known that $\Pi_{m,n}$ is given by
a determinant
\begin{equation}\label{me0}\diff\Pi_{m,n}
(\lambda_1,\ldots,\lambda_n)=
\frac {1} {n!}~
\det\lpar
K_{m,n}(\lambda_i,\lambda_j)~\e^{-m(Q(\lambda_i)+Q(\lambda_j))/2}\rpar_{i,j=1}^n
~\dA_n(\lambda_1,\ldots,\lambda_n).\end{equation}

More generally, for $k\le n$, the \textit{$k$-point marginal distribution}
$\Pi_{m,n}^k$ is the probability measure on $\C^k$ which is
characterized by
\begin{equation}\label{marginal}
\int_{\C^k}f(\lambda_1,\ldots,\lambda_k)~\diff \Pi_{m,n}^k
(\lambda_1,\ldots,\lambda_k)=
\int_{\C^n}f(\lambda_{\pi(1)},\ldots,\lambda_{\pi(k)})~\diff
\Pi_{m,n}(\lambda_1,\ldots,\lambda_n),
\end{equation}
whenever $f$ is a continuous bounded function depending only on $k$
variables and $\pi:\{1,\ldots,k\}\to \{1,\ldots,n\}$ is injective.
Evidently, $\Pi_{m,n}=\Pi_{m,n}^n$. One then has that
\begin{equation}\label{pmnk}\diff \Pi_{m,n}^k(\lambda_1,\ldots,\lambda_k)=
\frac {(n-k)!} {n!}
\det\lpar
K_{m,n}(\lambda_i,\lambda_j)~\e^{-m(Q(\lambda_i)+Q(\lambda_j))/2}\rpar_{i,j=1}^k
\dA_k(\lambda_1,\ldots,\lambda_k).\end{equation}

For proofs of the identities
 \eqref{me0}, \eqref{pmnk}, see e.g. \cite{M}, \cite{HM}, or the argument in \cite{ST}, \textsection IV.7.2.

The weighted kernel $K_{m,n}(z,w)~\e^{-m(Q(z)+Q(w))/2}$ is
known as the \textit{correlation kernel} or \textit{Christoffel--Darboux kernel} corresponding to the ensemble.

\subsection*{Linear statistics}
Let us now fix a function $g\in\calC_b(\C)$ and
form the random variable ("linear statistic'')
\begin{equation*}\trace_n g:\C^n\to\C\quad ,\quad (\lambda_j)_{j=1}^n\mapsto
\sum_{j=1}^n g(\lambda_j).\end{equation*} Let $E_{m,n}$ denote
expectation with respect to the measure $\Pi_{m,n}$ on $\C^n$. Likewise, if
$k\le n$ we let $E_{m,n}^k$ denote expectation with respect to the
marginal distribution $\Pi_{m,n}^k$. Then by \eqref{marginal}, \eqref{pmnk}
\begin{equation}\label{E1.8}E_{m,n}\lpar \frac 1 n\trace_n g\rpar=
\frac 1 n \sum_{j=1}^n
E_{m,n}^1(g(\lambda_j))=E_{m,n}^1(g(\lambda_1))=\frac 1 n \int_\C
g(\lambda_1)\, K_{m,n}(\lambda_1,\lambda_1)~\e^{-mQ(\lambda_1)}
\dA(\lambda_1).
\end{equation}
The asymptotics of the right hand side
can be deduced from the following fact (see \cite{HM}, cf. also
\cite{B},\cite{EF},\cite{EM})
\begin{equation}\label{concl}\int_\C\babs{~\frac 1 n~ K_{m,n}(\lambda,\lambda)~\e^{-mQ(\lambda)}-
\tau^{-1}\Delta\widehat{Q}_\tau(\lambda)~}~\dA(\lambda)\to 0,\quad \text{as}\quad m\to\infty,\, n/m\to\tau.\end{equation}
Combining with \eqref{E1.8} one obtains the following well-known result.

\begin{thm} (\cite{HM}) Let $g\in\calC_b(\C)$. Then
\begin{equation*}\frac 1 nE_{m,n}\lpar \trace_n g\rpar\to
\int_\C g(\lambda)~\diff\sigma_\tau(\lambda),\quad \text{as} \quad
m\to\infty,\quad n/m\to\tau.
\end{equation*}
\end{thm}

We now form the random variable ("fluctuation about the equilibrium''),
\begin{equation*}\fluct_n g=
\trace_n g-n\int_\C g~\diff \sigma_\tau.\end{equation*}

The main problem considered in this paper is to determine the
asymptotic distribution of $\fluct_n g$ as $m\to\infty$ and
$n-m\tau\to 0$ when $g$ is supported in the interior
("bulk'') of
$\setS_\tau\cap X$. For this purpose, we will use a
result, due to Berman \cite{B}, concerning the near-diagonal bulk asymptotics
of the correlation kernel.

\subsection*{Approximating Bergman kernels}
For convenience we assume that $Q$ be \textit{real-analytic} in neighbourhood of the droplet. This is not a serious restriction, see Remark \ref{r1.9} and \textsection \ref{nap}. (Moreover, the real analytic case is the most
interesting one.)

Let $b_0(z,w)$, $b_1(z,w)$ and $\psi(z,w)$ be the (unique)
holomorphic functions defined in a neighbourhood in $\C^2$ of the set
$\left\{(z,\bar{z});z\in \setS_\tau\cap X\right\}$ such that $b_0(z,\bar{z})=\Delta Q(z)$,
$b_1(z,\bar{z})=\frac 1 2\Delta \log\Delta Q(z)$, and
$\psi(z,\bar{z})=Q(z)$ for all $z\in X$.
The \textit{first-order approximating
Bergman kernel} $K_m^1(z,w)$ is defined by
\begin{equation*}K_m^1(z,w)=\lpar mb_0(z,\bar{w})+b_1(z,\bar{w})\rpar~
\e^{m\psi(z,\bar{w})},\end{equation*} for all $z,w$ where it makes
sense, viz. in a neighbourhood of the anti-diagonal
$\left\{(z,\bar{z});~z\in \setS_\tau\cap X\right\}$.

\begin{lem}\label{klam} (\cite{B})
Let $K$ be a compact subset of $\setS_\tau^\circ\cap X$, and fix
$z_0\in K$. There then exists a numbers $m_0$, $C$ and $\eps>0$
independent of $z_0$ such that for all $m\ge m_0$ holds
\begin{equation*}\babs{~K_{m,n}(z,w)-K_m^1(z,w)~}~\e^{-m(Q(z)+Q(w))/2}\le
Cm^{-1},\quad z,w\in D(z_0;\eps),\quad n\ge m\tau-1.\end{equation*}
In particular,
\begin{equation}\label{popp}\babs{~K_{m,n}(z,z)~\e^{-mQ(z)}-
\lpar
m\Delta Q(z)+\frac 1 2\Delta\log\Delta Q(z)\rpar~} \le Cm^{-1},\quad z\in K,
\quad n\ge m\tau-1.
\end{equation}
\end{lem}

A proof of the result in the present form appears in \cite{AH}, Theorem 2.8, using essentially the method of Berman \cite{B} and the approximate Berman projections
constructed in \cite{BBS}
(compare also \cite{BSZ}, \cite{BSZ0}). Cf. \cite{B}, \textsection 1.3 for a comparison with the line bundle setting.

We remark that corresponding uniform estimates in Lemma \ref{klam}, up to the boundary of the droplet, are false.

\subsection*{Expectation of fluctuations} Using Lemma \ref{klam}, we can easily prove the following result.

\begin{thm} \label{t0}Suppose that $g\in \coity(\setS_\tau^\circ\cap X)$. Then
\begin{equation*}E_{m,n}\fluct_n g\to
\int_\C g~\diff\nu\qquad \text{as}\quad m\to\infty\quad
\text{and}\quad n-m\tau\to 0,
\end{equation*}
where $\nu$ is the signed measure
\begin{equation*}\diff\nu(z)=\frac 1 2 \Delta\log\Delta Q(z) ~\1_{\setS_\tau\cap X}(z)~\diff
A(z).\end{equation*}
\end{thm}

\begin{proof} By \eqref{popp},
\begin{equation*}\begin{split}&
E_{m,n}(\fluct_n g)=nE_{m,n}^1 g(\lambda_1)-n\int_\C
g(\lambda_1)~\diff\sigma_\tau(\lambda_1)=\\
&=\int_{\supp g}\lpar m\Delta Q(z)+\frac 1 2 \Delta\log\Delta Q(z)+
\calO(m^{-1})\rpar
g(z)~\dA(z)-n\tau^{-1}\int_{\supp g} g(z)~\Delta Q(z)~\dA(z)=\\
&=(m-n\tau^{-1})\int g(z)~\Delta Q(z)~\dA(z)+\frac 1 2 \int
g(z)~\Delta\log\Delta
Q(z)~\dA(z)+\calO(m^{-1}).\\
\end{split}
\end{equation*}
When $m\to\infty$ and $m-n\tau^{-1}\to 0$, the expression in the
right hand side converges to $\int_\C g\,\diff\nu.$
\end{proof}

\subsection*{Main result} Let $\nabla=(\d/\d x,\d/\d y)$ denote the usual gradient on $\C=\R^2$. We have the following theorem.

\begin{thm} \label{mthm} Let $g\in\coity\lpar\setS_\tau^\circ\cap X\rpar$.
The random variable $\fluct_n g$ on the probability space
$(\C^n,\Pi_{m,n})$ converges in distribution when $m\to\infty$ and
$n-m\tau\to 0$ to a Gaussian variable with expectation $e_g$ and
variance $v_g^2$ given by
\begin{equation*}e_g=\int g~\diff\nu\quad ,\quad
v_g^2=\frac 1 4\int\babs{~\nabla g~}^2\dA.\end{equation*}
\end{thm}

This theorem is the analog of a result due to Johansson
\cite{J}, where the Hermitian case is considered. Following Johansson we note that in contrast to situation of the standard CLT, there is no $1/\sqrt{n}$-normalization of the
fluctuations. The variance is thus very small compared to what it would be in the i.i.d. case. This means that there must be effective cancelations,
caused by the repulsive behaviour
of the eigenvalues. One can interpret Theorem \ref{mthm}
as the statement that the random distributions $\fluct_n$
converge to a \textit{Gaussian field} on compact subsets
of the bulk of the droplet. See \textsection\ref{interp}.

The formula for $e_g$ has already been shown. The rest of the paper
is devoted to proving the other statements, viz. the formula for
$v_g^2$ and the asymptotic normality of the variables $\fluct_n g$
when $m\to \infty$ and $n-m\tau\to 0$. A simple argument shows that
it suffices to show these properties for \textit{real-valued}
functions $g$. In the following sections we will hence assume that
$g$ is real-valued.

We will give two proofs of Theorem \ref{mthm}, one using the well-known
cumulant method will be given in detail and another using an idea of Johansson \cite{J} is sketched in \textsection\ref{var}.
 A comparison is found in Remark \ref{end}.

We here want to mention the parallel work by Berman \cite{B2}, who independently gave a different proof of a version Th.~\ref{mthm} valid in a more general
situation involving several complex variables.

\medskip

\begin{rem} \label{r1.9} We emphasize that in our first,
cumulant-based proof of Th.~\ref{mthm}
we assume that $Q$ be real
analytic in a neighbourhood of the droplet. The theorem is however true e.g. also for general $Q:\C\to \R\cup\{+\infty\}$ which satisfy \eqref{gro} and are
finite and $\calC^\infty$ except
in a finite set where the value is $+\infty$. Since this type of potentials are sometimes useful, we will after the proof
indicate
the modifications needed to make it work in this generality.
See \textsection \ref{nap}.
\end{rem}

\subsection*{The cumulant method} For a real-valued random variable $A$, the cumulants $\calC_k(A)$,
$k\ge 1$, are defined by
\begin{equation}\label{kumu}\log {\mathbf E}\lpar \e^{ tA}\rpar=\sum_{k=1}^\infty \frac {
t^k} {k!}~\calC_k(A),\end{equation} and $A$ is Gaussian if and only
if $\calC_k(A)=0$ for all $k\ge 3$. Moreover, $\calC_2(A)$ is the
variance of $A$.

It was observed by Marcinkiewicz that in order to prove asymptotic normality of a sequence of r.v.'s (i.e. convergence in distribution to a normal distribution), it suffices to
prove convergence of all moments, or, equivalently, convergence of the cumulants. Indeed convergence of the moments
is somewhat stronger than asymptotic normality.

We now fix a real-valued function
$g\in\coity(\setS_\tau^\circ\cap X)$ and write $\calC_{m,n,k}(g)$
for the $k$'th cumulant of $\trace_n g$ with respect to the measure
$\Pi_{m,n}$. Following Rider--Virág \cite{RV}, we can write
the cumulants as integrals involving the cyclic product
\begin{equation}\label{rmnkdef}R_{m,n,k}(\lambda_1,\ldots,\lambda_k)=
K_{m,n}(\lambda_1,\lambda_2)~K_{m,n}(\lambda_2,\lambda_3)~
\cdots~
K_{m,n}(\lambda_k,\lambda_1)~\e^{-m(Q(\lambda_1)+
\ldots+Q(\lambda_k))}.
\end{equation}
Namely, with
\begin{equation}\label{gkdef}
G_k(\lambda_1,\ldots,\lambda_k)=\sum_{j=1}^k\frac {(-1)^{j-1}} j
\sum_{k_1+\ldots+k_j=k, k_1,\ldots,k_j\ge 1} \frac {k!} {k_1!\cdots
k_j!}~\prod_{l=1}^j g(\lambda_l)^{k_l},
\end{equation}
we have (\cite{RV}, cf. also \cite{CL}, \cite{S1}, \cite{S2})
\begin{equation}\label{fund}\calC_{m,n,k}(g)=\int_{\C^k}G_k(\lambda_1,\ldots,\lambda_k)~
R_{m,n,k}(\lambda_1,\ldots,\lambda_k)~\dA_k(\lambda_1,\ldots,\lambda_k).\end{equation}
Note that if $G_k(\lambda_1,\ldots,\lambda_k)\ne 0$, then
$\lambda_i\in \supp g$ for some $i$.

The representation \eqref{fund} was used by Rider and Virág \cite{RV} in the case of the \textit{Ginibre potential} $Q=\babs{z}^2$ to prove the
desired convergence of the cumulants. In another paper \cite{RV0}, the same authors applied the cumulant method
to study some determinantal processes in the model Riemann surfaces, and they prove analogs of Th.~\ref{mthm} for a few other special (radial) potentials.

The methods of \cite{RV}, \cite{RV0} depends on the explicit form of the correlation kernel. In the present case, the explicit kernel is too complicated to be of much use. To circumvent this problem we will use
the asymptotics in Lemma \ref{klam}, and also some off-diagonal damping results for the correlation kernels (cf. section \ref{poh}).

We want to emphasize that the
result of \cite{RV} covers also the situation when $g$ is not
necessarily supported in the bulk (in Ginibre case), and this situation is not treated in Th.~\ref{mthm}. (We shall have more to say about that case in general
in \textsection\ref{flubou} below.)

The cumulant method is well-known and has been used earlier e.g. by Soshnikov \cite{S1} and Costin--Lebowitz \cite{CL} to obtain results on asymptotic normality of fluctuations of linear statistics of eigenvalues
from some classical compact groups. The method has also been used in the parallel work on linear statistics of zeros of Gaussian analytic functions
initiated by Sodin--Tsirelson \cite{STs} and generalized by Shiffman--Zelditch \cite{SZ}. A brief comparison of these results to those of the present paper is given
in \textsection\ref{compare}.

\subsection*{Other related work}
It should also be noted that Th.~\ref{mthm}, as well as the more
general Th.~\ref{grth} below, follow from the well-known "physical''
arguments due to Wiegmann et al. See e.g. the survey \cite{Z} and
the references therein as well as \cite{F}.

Results related to fluctuations of eigenvalues of
Hermitian matrices, are found in Johansson \cite{J} and also \cite{AZ}, \cite{BS} and \cite{G}. A lot of work has been done concerning ensembles
connected with the classical compact groups. See e.g.
\cite{DE}, \cite{Jo}, \cite{S1}, \cite{W}, \cite{CL}.

\subsection*{Disposition and further results} Sections \ref{FACT}--\ref{puuh} comprise our cumulant-based proof of Theorem \ref{mthm}.
In our concluding remarks section, Sect. \ref{conclu}, we state and prove further results. We summarize some of them here. In \textsection \ref{var} we sketch an alternative proof of Th.~\ref{mthm} based on a variational approach in the spirit of Johansson's paper
\cite{J}. In \textsection\ref{flubou}, we state without proof the full plane version of Th.~\ref{mthm}. The proof will appear in \cite{AHM2}. In \textsection\ref{ginibre} we prove universality under the natural scaling: if $m=n$ then for a fixed $z_0\in\setS_1^\circ\cap X$, the rescaled point process
$\lpar \lambda_j\rpar_{j=1}^n\mapsto \lpar \sqrt{n}(\lambda_j-z_0)\rpar_{j=1}^n$ converges to the Ginibre$(\infty)$ determinantal point process as $n\to\infty$.
In \textsection\ref{skiit} we clarify the relation of our present results to the Berezin transform (which we studied in \cite{AH}); in particular we prove the "wave-function conjecture" (\cite{HM}) that $\babs{P_n}^2\e^{-nQ}\diff A$ converges to harmonic measure at $\infty$ with respect to $\hat{\C}\setminus \setS_1$, where $P_n$ is the $n$:th orthonormal polynomial corresponding to the weight $\e^{-nQ}$ and $\hat{\C}=\C\cup\{\infty\}$.

\section{Further approximations and consequences of Taylor's formula.}\label{FACT}

In this preparatory section, we discuss a variant of the
near-diagonal bulk asymptotics for the correlation kernel and for the functions $R_{n,m,k}$ (see \eqref{rmnkdef}), especially for $k=2,3$; such estimates are easily obtained by inserting
the asymptotics in Lemma \ref{klam}, and they will be used
in Sect. \ref{puuh}.

In this and the following sections, we assume that $Q$ is real analytic near the droplet, except when otherwise is specified. Recall that $\psi$ denotes the holomorphic extension
of $Q$ from the anti-diagonal, i.e. $\psi(z,\bar{z})=Q(z)$.

 It is well-known and easy to show that $\psi$ is determined in a neighbourhood of a
point at the anti-diagonal by the series
\begin{equation*}\begin{split}&\psi\lpar z+h,\overline{z+k}\rpar=
\sum_{i,j=0}^\infty \d^i\dbar^jQ(z)\frac
{h^i\bar{k}^j}{i!j!},\\
\end{split}\end{equation*}
for $h$ and $k$ in a neighbourhood of $0$.

For clarity of the exposition, it is here worthwhile to explicitly
write down the first few terms in the series for $\psi$ and $Q$
\begin{equation*}\begin{split}
&\psi\lpar z+h,\overline{z+k}\rpar=\\
&=Q(z)+\d Q(z)~h+\dbar Q(z)~\bar{k}+\frac 1 2 \lpar \d^2
Q(z)~h^2+\dbar^2 Q(z)~\bar{k}^2\rpar +\Delta
Q(z)~h\bar{k}+\text{"higher
order terms"},\\
\end{split}
\end{equation*}
and
\begin{equation*}\begin{split}&Q(z+h)
=Q(z)+\d Q(z)~h+\dbar Q(z)~\bar{h}+\frac 1 2\lpar \d^2
Q(z)~h^2+\dbar^2 Q(z)~\bar{h}^2\rpar +\Delta
Q(z)~\babs{h}^2+\calO(\babs{h}^3),\\
\end{split}
\end{equation*}
for small $\babs{h}$.
Using that $\overline{\psi(z,w)}=\psi(\bar{w},\bar{z})$, and that $Q$ is real-analytic near the droplet, it is easy
to prove uniformity of the $\calO$-terms in $z$ when $z\in\setS_\tau$.
This means that there is
$\eps>0$ such that
\begin{equation}\label{vis2}\babs{~2\re \psi(z+h,\bar{z})-Q(z)-Q(z+h)+\Delta
Q(z)\babs{h}^2~}\le C\babs{h}^3,\quad z\in \setS_\tau,\,
\babs{h}\le\eps.\end{equation}

We will in the following consider $h$ such that $\babs{h}\le M\delta_m$ where $M$ is
fixed and
\begin{equation*}\delta_m=\log m/\sqrt{m}.\end{equation*}
We then infer from \eqref{vis2} that there is a number $C$ depending
only $M$ such that
\begin{equation*}\babs{~2m\re\psi(z+h,\bar{z})-mQ(z)-mQ(z+h)+m\Delta Q(z)\babs{h}^2~}
\le Cm\delta_m^3,\quad z\in \setS_\tau,\, \babs{h}\le
M\delta_m,\end{equation*} and $m\delta_m^3=\log^3 m/\sqrt{m}\to 0$
when $m\to\infty$.

Next recall the definition of the approximating kernel $K_m^1(z,w)=\lpar mb_0(z,\bar{w})+b_1(z,\bar{w})\rpar \e^{m\psi\lpar z,\bar{w}\rpar}$ (see Lemma \ref{klam}).
We obviously have
\begin{equation}\label{1boexp}
\babs{b_0(z+h,\bar{z})-\Delta Q(z)}\le C\delta_m\quad
\text{and}\quad \babs{b_1(z+h,\bar{z})}\le C\quad \text{when} \quad
z\in \setS_\tau,\, \babs{h}\le M\delta_m,
\end{equation}
for all large $m$ with $C$ depending only on $K$ and $M$. It follows
that
\begin{equation}\label{boexp}\begin{split}K_m^1(z+h,z)~\e^{-m(Q(z+h)+Q(z))/2}
%&=
%\lpar
%mb_0(z+h,\bar{z})+b_1(z+h,\bar{z})\rpar~\e^{m(\psi(z+h,\bar{z})-(Q(z)+Q(z+h))/2)}=\\
&=m(\Delta Q(z)+\calO(\delta_m))~\e^{m(\psi(z+h,\bar{z})-(Q(z)+Q(z+h))/2)},\quad z\in \setS_\tau,\,
\babs{h}\le M\delta_m,\\
\end{split}\end{equation} when $m\to \infty$. Here the $\calO$-term is
uniform in $z\in \setS_\tau$. Lemma \ref{klam} now implies the following
estimate for the correlation kernel.

\begin{lem}\label{prev} Fix a compact subset $K\subset\setS_\tau^\circ\cap X$. Then
for all $z\in K$ we have that
\begin{equation*}\begin{split}\babs{K_{m,n}(z+h,z)}&~\e^{-m(Q(z+h)+Q(z))/2}=\\
&= m\lpar \Delta Q(z)+\calO(\delta_m)\rpar~\e^{-m\Delta
Q(z)\babs{h}^2/2+\calO(\log^3 m/\sqrt{m})}+\calO(m^{-1}), \quad
\babs{h}\le M\delta_m,\\
\end{split}\end{equation*} when $m\to\infty$ and $n\ge m\tau-1$; the
$\calO$-terms are uniform in $z$ for $z\in K$.
\end{lem}

We will need a consequence concerning the functions $R_{m,n,k}$ for $k=2$ and $k=3$.

\begin{lem} Let $K$ be a compact subset of $\setS_\tau^\circ\cap X$.
Then for $z\in K$ we have
\begin{equation}\label{apu}R_{m,n,2}(z,z+h)=m^2
\lpar \Delta Q(z)^2+\calO(\delta_m)\rpar~\e^{-m\Delta Q(z)\babs{h}^2+
\calO(\log^3 m/\sqrt{m})}+\calO(1), \quad \babs{h}\le
M\delta_m,\end{equation} and
\begin{equation}\label{apuu}\begin{split}R_{m,n,3}&(z,z+h_1,z+h_2)=\\
&=m^3(\Delta Q(z)^3+\calO(\delta_m))~\e^{m\Delta Q(z)\lpar
h_1\bar{h}_2-\babs{h_1}^2-\babs{h_2}^2\rpar+\calO(\log^3
m/\sqrt{m})}+\\
&\quad+ \calO\lpar 1+m\lpar \e^{-m\Delta Q(z)\babs{h_1}^2/2}+\e^{-m\Delta Q(z)\babs{h_2}^2/2}\rpar\rpar,\, \babs{h_1},\babs{h_2}\le
M\delta_m,\\
\end{split}\end{equation} when $m\to\infty$ and $n\ge m\tau-1$; the
$\calO$-terms are uniform for $z\in K$.
\end{lem}

\begin{proof}
The estimate \eqref{apu} follows from Lemma
\ref{prev} since
$R_{m,n,2}(z,z+h)=\babs{K_{m,n}(z,z+h)}^2
~\e^{-m(Q(z)+Q(z+h))}$.

To estimate $R_{m,n,3}(z,z+h_1,z+h_2)$  we first consider the
approximation
\begin{equation*}R_{m,3}^1(z,z+h_1,z+h_2)=K_{m}^1(z,z+h_1)K_{m}^1(z+h_1,z+h_2)
K_{m}^1(z+h_2,z)~\e^{-m(Q(z)+Q(z+h_1)+Q(z+h_2))},\end{equation*}
obtained by replacing $K_{m,n}$ by $K_m^1$ in the definition of
$R_{m,n,3}$.

In view of \eqref{1boexp} we have for $z\in K$ and
$\babs{h_1},\babs{h_2}\le M\delta_m$ that
\begin{equation}\label{apa}R_{m,3}^1(z,z+h_1,z+h_2)=m^3\lpar \Delta
Q(z)^3+\calO(\delta_m)\rpar ~\e^{m(\psi(z,\overline{z+h_1})+\psi(z+h_1,\overline{z+h_2})+
\psi(z+h_2,\bar{z})-Q(z)-Q(z+h_1)-Q(z+h_2))},\end{equation} where
$\calO$ is uniform in $z\in K$. A simple
calculation with the Taylor expansions for $Q$ at $z$ and $\psi$ at
$(z,\bar{z})$ now yields that
\begin{equation*}\begin{split}
\psi(z,\overline{z+h_1})&+\psi(z+h_1,\overline{z+h_2})+
\psi(z+h_2,\bar{z})- Q(z)-Q(z+h_1)-Q(z+h_2)=\\
&=\Delta Q(z)\lpar
h_1\bar{h}_2-\babs{h_1}^2-\babs{h_2}^2\rpar+\calO(|h|_\infty^3),\quad \text{as}\quad h\to 0,\\
\end{split}\end{equation*}
where we have put $\babs{h}_\infty=\max\{\babs{h_1},\babs{h_2}\}$.
Since the estimate is uniform for $z\in K$, we may use \eqref{apa}
to conclude that
\begin{equation*}R_{m,3}^1(z,z+h_1,z+h_2)=m^3\lpar \Delta
Q(z)^3+\calO(\delta_m)\rpar~\e^{m\Delta Q(z)\lpar
h_1\bar{h}_2-\babs{h_1}^2-\babs{h_2}^2\rpar+\calO(\log^3
m/\sqrt{m})},\quad \babs{h}_\infty \le M\delta_m,\end{equation*}
when $m\to\infty$ and again the $\calO$-terms are uniform for $z\in
K$. Combining with Lemma \ref{prev} and \eqref{boexp}, and also using the estimate $\babs{K_{m,n}(z+h_1,z+h_2)}\e^{-m(Q(z+h_1)+Q(z+h_2))/2}\le Cm$ for $\babs{h}_\infty\le M\delta_m$, $n\ge m\tau-1$, $m$ large (this follows
from Lemma \ref{klam}) we readily obtain 
\eqref{apuu}.
\end{proof}

\section{The functions $G_k$; near-diagonal behaviour.}\label{GK}
\label{ssss} In this section, we let $g$ be any sufficiently smooth
(sometimes real-valued) function on $\C$ (i.e. not necessarily
supported in $\setS_\tau^\circ\cap X$). We then form the
corresponding function $G_k$ by \eqref{gkdef}. Here $k\ge 2$ is
fixed.

 We will now analyze the function $G_k$ in a neighbourhood of the diagonal
\begin{equation*}\triangle_k=\{\lambda\1_k \in\C^k;\lambda\in\C\},\end{equation*}
where
\begin{equation*}\1_k=(1,1,\ldots,1)\in\C^k.\end{equation*}

Our results in this section state that $G_k$
vanishes identically on $\triangle_k$ and that $G_k$ is harmonic at
each point of $\triangle_k$. This depends on combinatorial identities of a type which where considered earlier in related contexts e.g. by
Soshnikov \cite{S1} and Rider--Virág \cite{RV}, \cite{RV0}. The following lemma
is equivalent to \cite{S1}, eq. (1.14), p. 1356.

\begin{lem} \label{zerodiag} For any function
$g:\C\to \C$ and any $k\ge 2$, it holds that $G_k=0$ on
$\triangle_k$.
\end{lem}

\begin{proof} Evidently
\begin{equation*}G_k(\lambda\1_k)=g(\lambda)^k\sum_{j=1}^k
\frac {(-1)^{j-1}} j \sum_{k_1+\ldots+k_j=k, k_1,\ldots,k_j\ge 1}
\frac {k!} {k_1!\cdots k_j!}.\end{equation*} The last sum is the
number of partitions of $k$ distinguishable elements into $j$
distinguishable, nonempty subsets. Thus (e.g. \cite{C}, Th.~9.1, p.
340)
\begin{equation*}\sum_{k_1+\cdots+k_j=k, k_1,\ldots,k_j\ge 1}
\frac {k!} {k_1!\cdots k_j!}=j!S(k,j),
\end{equation*}
where
\begin{equation*}S(k,j)=\frac 1 {j!}\sum_{r=0}^j (-1)^r {j\choose r}(j-r)^k
\end{equation*}
is the Stirling number of the second kind. Evidently $S(k,0)=0$ for
$k\ge 1$. Moreover, the well-known recurrence relation for those
Stirling numbers (see e.g. \cite{C}, Th.~8.9, (8.32)), gives
\begin{equation*}S(k-1,0)=\sum_{r=0}^{k-1} (-1)^rr!S(k,r+1)=\sum_{j=1}^k
\frac {(-1)^{j-1}} j j!S(k,j).\end{equation*} The lemma follows,
since $S(k-1,0)=0$ when $k\ge 2$.
\end{proof}

Note that the lemma is equivalent to that
\begin{equation}\label{zerosum}\sum_{j=1}^k
\frac {(-1)^{j-1}} j \sum_{k_1+\cdots+k_j=k, k_1,\ldots,k_j\ge 1}
\frac {1} {k_1!\cdots k_j!}=0,\quad k=2,3,\ldots.\end{equation}

We note the following simple, but rather useful consequence of Lemma
\ref{zerodiag}.

\begin{lem}\label{zerodiagp} Let $g\in \calC^1(\C\to \C)$ and $k\ge
2$. Then for all $\lambda\in\C$ holds
\begin{equation*}\sum_{i=1}^k (\d_i
G_k)(\lambda_1,\ldots,\lambda_k)\biggm|_{\lambda_1=\cdots=\lambda_k=\lambda}=
\sum_{i=1}^k (\dbar _i
G_k)(\lambda_1,\ldots,\lambda_k)\biggm|_{\lambda_1=\cdots=\lambda_k=\lambda}=0.\end{equation*}
\end{lem}

\begin{proof} By Lemma
\ref{zerodiag} we have that $G_k(\lambda\1_k)=0$, whence
\begin{equation*} 0=\frac \d {\d\lambda}
G_k(\lambda\1_k)=\sum_{i=1}^k (\d_i
G_k)(\lambda\1_k).\end{equation*} The statement about $\dbar$ is
analogous.
\end{proof}

We now turn to a more nontrivial fact. Let us denote by
\begin{equation*}\Delta_k=\d_1\dbar_1+\ldots+\d_k\dbar_k,\end{equation*}
the Laplacian on $\C^k$.

In the next lemma, we calculate $\Delta_k G_k$ at every
point of the diagonal $\triangle_k$ when $k\ge 2$. When $k\ge 3$,
we shall see that $\Delta_k G_k$ vanishes on the diagonal, which means that
$G_k$ is nearly harmonic close to the diagonal.

\begin{lem}\label{zerolap} Let $g\in\calC^2(\C\to\R)$ and $k\ge 2$.
Then for all $\lambda\in\C$ we have
\begin{equation*}(\Delta_2
G_2)(\lambda_1,\lambda_2)\biggm|_{\lambda_1=\lambda_2=\lambda}=\babs{\nabla
g(\lambda)}^2/2,\end{equation*} and
\begin{equation*}
(\Delta_k
G_k)(\lambda_1,\ldots,\lambda_k)\biggm|_{\lambda_1=\ldots=\lambda_k=\lambda}=0,\quad
k=3,4,\ldots\end{equation*}
\end{lem}

\begin{proof} Fix a number $k\ge 2$. Let $1\le j\le k$, and let $k_1,\ldots, k_j$ be
positive integers such that $k_1+\ldots+k_j=k$. Since, for $1\le
r\le j$,
\begin{equation*}\frac {\d^2} {\d \lambda_r \d \bar{\lambda_r}}\lpar
\prod_{l=1}^j g(\lambda_l)^{k_l}\rpar=k_r (k_r-1)\cdot \prod_{l=1,
l\ne r}^j g(\lambda_l)^{k_l}\cdot g(\lambda_r)^{k_r-2} \cdot \d
g(\lambda_r)\cdot \dbar g(\lambda_r)+ k_r\cdot \prod_{l=1, l\ne r}^j
g(\lambda_l)^{k_l}\cdot g(\lambda_r)^{k_r-1}\cdot \Delta
g(\lambda_r),
\end{equation*}
we get (with $\1_k=(1,\ldots,1)\in\C^k$)
\begin{equation}\label{busch}\begin{split}
(\Delta_k G_k)(\lambda\1_k)&= \sum_{j=1}^k\frac {(-1)^{j-1}} j
\sum_{k_1+\ldots+k_j=k,k_1,\ldots,k_j\ge 1}\frac {k!}
{k_1!\cdots k_j!} \times\\
& \times \lpar g(\lambda)^{k-2}\babs{\dbar g(\lambda)}^2\sum_{r=1}^j
k_r(k_r-1)+ g(\lambda)^{k-1}\Delta g(\lambda)\sum_{r=1}^j k_r\rpar.\\
\end{split}
\end{equation}
Since $k_1+\ldots+k_j=k$, the right hand side in \eqref{busch}
simplifies to
\begin{equation}\label{f1}\begin{split}&g(\lambda)^{k-2}\babs{\dbar
g(\lambda)}^2 \sum_{j=1}^k\frac {(-1)^{j-1}} j
\sum_{k_1+\ldots+k_j=k,\, k_1,\ldots, k_j\ge 1}
\frac {k!(k_1(k_1-1)+\ldots+k_j(k_j-1))} {k_1!\cdots k_j!}+\\
&+g(\lambda)^{k-1}\Delta g(\lambda)\sum_{j=1}^k\frac {(-1)^{j-1}} j
\sum_{k_1+\ldots+k_j=k,\, k_1,\ldots, k_j\ge 1}\frac {k\cdot k!}
{k_1!\cdots k_j!}.\\
\end{split}
\end{equation}
Here the last double sum is zero, by \eqref{zerosum}, and \eqref{f1}
simplifies to
\begin{equation}\label{f2}g(\lambda)^{k-2}\babs{\dbar g(\lambda)}^2\sum_{j=1}^k\frac {(-1)^{j-1}} j
\sum_{k_1+\ldots+k_j=k,k_1,\ldots,k_j\ge 1} \frac
{k!(k_1(k_1-1)+\ldots+k_j(k_j-1))} {k_1!\cdots k_j!}.\end{equation}
In order to finish the proof we must thus show that $S_2=2$ and
$S_k=0$ for all $k\ge 3$ where $S_k$ denotes the sum
\begin{equation}\label{skdef}S_k=\sum_{j=1}^k\frac {(-1)^{j-1}} j
\sum_{k_1+\ldots+k_j=k,k_1,\ldots,k_j\ge 1} \frac
{k!(k_1(k_1-1)+\ldots+k_j(k_j-1))} {k_1!\cdots k_j!}.\end{equation}
The case $k=2$ is trivial, so we assume that $k\ge 3$. To this end,
we shall consider exponential generating functions of the form
\begin{equation}\label{rel}
H_j(t;x_1,\ldots,x_j)=\prod_{l=1}^j\lpar \e^{tx_l}-1\rpar
=\sum_{k_1=1}^\infty \frac {(x_1 t)^{k_1}} {k_1!}\cdots
\sum_{k_j=1}^\infty \frac {(x_j t)^{k_j}} {k_j!}.
\end{equation}
The relevance of this generating function is seen when we expand the
product as a power series in $t$,
\begin{equation*}H_j(t;x_1,\ldots,x_j)=\sum_{k=1}^\infty
\lpar \sum_{k_1+\ldots+k_j=k,k_1,\ldots,k_j\ge 1} \frac
{k!x_1^{k_1}\cdots x_j^{k_j}} {k_1!\cdots k_j!}\rpar \frac {t^k}
{k!}.
\end{equation*}
Considering the $x_j$:s as real variables and denoting
\begin{equation*}\Delta_{j}^\R=\frac {\d^2} {\d x_1^2}+\ldots+\frac {\d^2}
{\d x_j^2},\end{equation*} the Laplacian on $\R^j$, we thus obtain
\begin{equation}\label{f3}\Delta_{j}^\R H_j(t;1,\ldots,1)=
\sum_{k=1}^\infty\lpar \sum_{k_1+\ldots+k_j=k,k_1,\ldots,k_j\ge 1}
\frac {k!(k_1(k_1-1)+\ldots+k_j(k_j-1))} {k_1!\cdots k_j!}\rpar
\frac {t^k} {k!}.\end{equation} On the other hand, differentiating
the product in \eqref{rel} and evaluating at $x_1=\ldots=x_j=1$
yields
\begin{equation}\label{gur}\Delta_{j}^\R H_j(t;1,\ldots,1)=jt^2
\e^{t}(\e^t-1)^{j-1},\end{equation} Differentiating \eqref{f3} $k$
times with respect to $t$ and evaluating at $t=0$, we obtain the
result that
\begin{equation*}\sum_{k_1+\ldots+k_j=k,k_1,\ldots,k_j\ge 1}
\frac {k!(k_1(k_1-1)+\ldots+k_j(k_j-1))} {k_1!\cdots k_j!}= \frac
{\diff^k} {\diff t^k}\lpar jt^2\e^t\lpar \e^t-1\rpar^{j-1}
\rpar\biggm|_{t=0}.
\end{equation*}
In view of \eqref{skdef}, this implies that
\begin{equation}\label{ssoo}S_k=\frac {\diff^k} {\diff t^k}\lpar
\sum_{j=1}^k
(-1)^{j-1}t^2\e^t\lpar\e^t-1\rpar^{j-1}\rpar\biggm|_{t=0} =\frac
{\diff^k} {\diff t^k} \lpar t^2\lpar 1-\lpar
1-\e^t\rpar^k\rpar\rpar\biggm|_{t=0}.
\end{equation}
But since $1-\e^t=-(t+t^2/2!+t^3/3!+\ldots)$, it is seen that the
coefficients $a_l$ in the expansion
\begin{equation*}t^2\lpar 1-\lpar 1-\e^t\rpar^k\rpar=\sum_{l=0}^\infty
a_lt^l\end{equation*} must vanish whenever $l\ne 2$ and $l<k+2$. In
particular, if, as we have assumed, $k$ is at least $3$, then we
have $a_k=0$, which by \eqref{ssoo} implies that $S_k=0$. The proof
is finished.
\end{proof}

In addition to the Laplacian $(\Delta_k G_k)(\lambda\1_k)$, we will
also need to consider functions of the form
\begin{equation}\label{zkdef}Z_k(\lambda)=\sum_{i<j}(\d_i\dbar_j
G_k)(\lambda\1_k),\qquad k\ge 2.\end{equation} The following lemma
is now easy to prove.

\begin{lem}\label{zgr} We have that $Z_2(\lambda)=-\babs{\dbar g(\lambda)}^2$
while $Z_k$ is pure imaginary when $k\ge 3$.
\end{lem}

\begin{proof} Again the case $k=2$ is trivial because
$G_2(\lambda_1,\lambda_2)=g(\lambda_1)^2-g(\lambda_1)g(\lambda_2)$.
When $k\ge 3$ we may use lemmas \ref{zerodiag} and \ref{zerolap} to
calculate
\begin{equation*}0=\Delta_\lambda\{G_k(\lambda\1_k)\}=(\Delta_k
G_k)(\lambda\1_k)+\sum_{i\ne j}(\d_i\dbar_j G_k)(\lambda\1_k)=2\re
Z_k(\lambda),\end{equation*} which shows that $Z_k$ is pure
imaginary.
\end{proof}

\section{An expansion formula for the cumulants}

During this section, we keep a \textit{real valued} function
$g\in\coity(\setS_\tau^\circ\cap X)$ fixed. We will reduce the proof
of Th.~\ref{mthm} to the proof of another statement (Th.~\ref{mth}
below), which turns out to be easier to handle, and which we prove
in the Sect. \ref{puuh}, after a discussion of some basic estimates
for $K_{m,n}$ in Sect. \ref{poh}.

To get started, note that an expression for the cumulant
$\calC_{m,n,k}(g)$ was given above in eq. \eqref{fund}. It will be
important to note that \eqref{fund} and the reproducing property of
$K_{m,n}$ shows that we may also represent the cumulant
$\calC_{m,n,k}(g)$ as an integral over $\C^{k+1}$,
\begin{equation}\label{rewrite}
\calC_{m,n,k}(g)= \int_{\C^{k+1}}G_k(\lambda_1,\ldots,\lambda_k)~
R_{m,n,k+1}(\lambda,\lambda_1,\ldots,\lambda_k)~
\dA_{k+1}(\lambda,\lambda_1,\ldots,\lambda_k),\end{equation}
where $G_k$ and $R_{m,n,k+1}$ are given by \eqref{gkdef} and
\eqref{rmnkdef} respectively. Indeed, this simple trick of
introducing an extra parameter $\lambda$ into the integral will turn
out to be of fundamental importance for our proof.

In the foregoing section, we were able to give a good description of
$G_k(\lambda_1,\ldots,\lambda_k)$ for points near the diagonal
$\lambda_1=\ldots=\lambda_k=\lambda$. For such points it is natural
to write $h_i=\lambda_i-\lambda$ (where the $\babs{h_i}$ are small)
and to work in the coordinate system $(\lambda,h_1,\ldots,h_k)$.
Indeed, this coordinate system is advantageous for all our purposes.
Note that the volume element is invariant with respect to this
change of coordinates,
\begin{equation*}\dA_{k+1}(\lambda,\lambda_1,\ldots,\lambda_k)=\dA_{k+1}(\lambda,h_1,\ldots,h_k),\end{equation*}
and that the reproducing property of $K_{m,n}$ is reflected by the
fact that
\begin{equation*}u(\lambda)=\int_\C
u(h)~K_{m,n}(\lambda,\lambda+h)~\e^{-mQ(\lambda+h)}~\dA(h),\quad u\in
H_{m,n}.\end{equation*}

We thus get that with $h=(h_1,\ldots,h_k)$ and $\1_k=(1,\ldots,1)$,
we can write \eqref{rewrite} as
\begin{equation}\label{rw2}\calC_{m,n,k}(g)=
\int_{\C^{k+1}}G_k(\lambda\1_k+h)~
R_{m,n,k+1}(\lambda,\lambda\1_k+h)~\dA_{k+1}(\lambda,h).\end{equation}

We now fix $\lambda\in\C$ and use Taylor's formula applied to the
function \begin{equation*}\C^k\to \R\quad :\quad h\mapsto
G_k(\lambda\1_k+h).\end{equation*} Since $G_k(\lambda\1_k)=0$ by
Lemma \ref{zerodiag}, the Taylor series at $h=0$ can be written
\begin{equation}\label{t1}G_k(\lambda\1_k+h)\sim \sum_{j=1}^\infty
T_j(\lambda,h),\end{equation} where, in the multi-index notation,
\begin{equation*}T_j(\lambda,h)=\sum_{\babs{\alpha+\beta}=j}\lpar \d^\alpha\dbar^\beta
G_k\rpar (\lambda\1_k)\frac {h^\alpha\bar{h}^\beta}
{\alpha!\beta!}.\end{equation*} Note that if $\lambda\not\in \supp
g$, then $G_k$ vanishes identically in a neighbourhood of
$\lambda\1_k$, and so $T_j(\lambda,h)=0$ for all $h\in\C^k$. Thus
the right hand side in \eqref{t1} is identically zero when
$\lambda\not\in\supp g$.

Let us write $|h|_\infty=\max\{\babs{h_1},\ldots,\babs{h_k}\}$. It
will turn out to be sufficient to consider Taylor series of degree
up to two. We thus put
\begin{equation}\label{t2}G_k(\lambda\1_k+h)=T_1(\lambda,h)+T_2(\lambda,h)+r(\lambda,h),
\quad \text{where}\quad r(\lambda,h)=\calO(\babs{h}_\infty^3)\quad
\text{as}\quad h\to 0.\end{equation} The idea is now to replace
$G_k(\lambda\1_k+h)$ by the right hand side in \eqref{t2} in the
integral \eqref{rw2}. To simplify matters, we first have the
following lemma.

\begin{lem} For all $k\ge 2$ holds
\begin{equation*}\int_{\C^{k+1}}T_1(\lambda,h)~
R_{m,n,k+1}(\lambda,\lambda\1_k+h)~\dA_{k+1}(\lambda,h)=0.\end{equation*}
\end{lem}

\begin{proof} First note that
\begin{equation}\label{int0}T_1(\lambda,h)=2\re \sum_{i=1}^k (\d_i
G_k)(\lambda\1_k)h_i.\end{equation} Integrating termwise in
\eqref{int0} with respect to the measure
$R_{m,n,k+1}(\lambda,\lambda\1_k+h)\dA_{k+1}(\lambda,h)$ and
observing that the terms on the right hand side of \eqref{int0}
depends only on two variables, the reproducing property of $K_{m,n}$
shows that, for $i=1,\ldots,k$,
\begin{equation*}\begin{split}\int_{\C^{k+1}}(\d_i
G_k)(\lambda\1_k)~R_{m,n,k+1}(\lambda,\lambda\1_k+h)~
h_i~\dA_{k+1}(\lambda,h)=\int_{\C^2}(\d_i
G_k)(\lambda\1_k)~R_{m,n,2}(\lambda,\lambda+h_1)~h_1~\dA_2(\lambda,h_1),\\
\end{split}\end{equation*} and so we can replace the integral in \eqref{int0}
by an integral over $\C^2$ (since $R_{m,n,2}$ is real-valued):
\begin{equation*}\begin{split}&\int_{\C^{k+1}}T_1(\lambda,h)R_{m,n,k+1}(\lambda,\lambda\1_k+h)\dA_{k+1}(\lambda,h)=
2\re \int_{\C^2}\lpar \sum_{i=1}^k (\d_i G_k)(\lambda\1_k)\rpar
R_{m,n,2}(\lambda,\lambda+h_1)h_1\dA_2(\lambda,h_1).\\
\end{split}
\end{equation*}
The last integral vanishes by Lemma \ref{zerodiagp}.
\end{proof}

We have shown now shown that
\begin{equation*}\calC_{m,n,k}(g)=\int_{\C^{k+1}}\lpar
T_2(\lambda,h)+r(\lambda,h)\rpar~
R_{m,n,k+1}(\lambda,\lambda\1_k+h)~\dA_{k+1}(\lambda,h).\end{equation*}

To simplify this expression further, we will first look more closely
at
\begin{equation*}T_2(\lambda,h)=\sum_{\babs{\alpha+\beta}=2}\lpar \d^\alpha\dbar^\beta
G_k\rpar (\lambda\1_k) \frac {h^\alpha\bar{h}^\beta}
{\alpha!\beta!},\end{equation*} which we write in the form
\begin{equation*}\begin{split}T_2(\lambda,h)&=
\frac 1 2\sum_{i,j=1}^k(\d_i\d_j
G_k)(\lambda\1_k)h_ih_j+
\frac 1 2 \sum_{i,j=1}^k (\dbar_i\dbar_j)G_k(\lambda\1_k)\bar{h}_i\bar{h}_j
+\sum_{i,j=1}^k (\d_i\dbar_j G_k)(\lambda\1_k)h_i\bar{h}_j=\\
&=\re \sum_{i=1}^k (\d_i^2 G_k)(\lambda\1_k)h_i^2+\re \sum_{i\ne j}
(\d_i\d_j G_k)(\lambda\1_k)h_i h_j+\\
&+\sum_{i=1}^k(\d_i\dbar_i G_k)(\lambda\1_k)\babs{h_i}^2+
2\re\sum_{i<j}(\d_i\dbar_j G_k)(\lambda\1_k)h_i\bar{h}_j.\\
\end{split}
\end{equation*}
Using the reproducing property of $K_{m,n}$, it yields (note that
$R_{m,n,k}$ is \textit{not real-valued} if $k\ge 3$)
\begin{equation*}\begin{split}&\int_{\C^{k+1}}T_2(\lambda,h)
R_{m,n,k+1}(\lambda,\lambda\1_k+h)~\dA_{k+1}(\lambda,h)=\\
&=\int_{\C^3}\re\lpar \sum_{i\ne j} (\d_i\d_j
G_k)(\lambda\1_k)h_1h_2\rpar
R_{m,n,3}(\lambda,\lambda+h_1,\lambda+h_2)~\dA_3(\lambda,h_1,h_2)+\\
&+\re \int_{\C^2}\lpar \sum_{i=1}^k (\d_i^2 G_k)(\lambda\1_k)\rpar
h_1^2 R_{m,n,2}(\lambda,\lambda+h_1)~\dA_2(\lambda,h_1)+
\\
&+2\int_{\C^3}\re\lpar \sum_{i< j} (\d_i\dbar_j G_k)(\lambda\1_k)
h_1\bar{h}_2\rpar R_{m,n,3}(\lambda,\lambda+h_1,\lambda+h_2)~\dA_3(\lambda,h_1,h_2)+\\
&+\int_{\C^2}\sum_{i=1}^k \lpar (\d_i\dbar_i
G_k)(\lambda\1_k)\rpar\babs{h_1}^2
R_{m,n,2}(\lambda,\lambda+h_1)~\dA_2(\lambda,h_1).\\
\end{split}
\end{equation*}
Let us now introduce some notation. Recall that
\begin{equation*}(\Delta_k G_k)(\lambda\1_k)=\sum_{i=1}^k
(\d_i\dbar_i G_k)(\lambda\1_k)\qquad \text{and}\qquad
Z_k(\lambda)=\sum_{i<j}(\d_i\dbar_j G_k)(\lambda\1_k),\qquad
\lambda\in\C.\end{equation*}

\begin{defn} Let us put
\begin{equation*}\begin{split}A_{m,n}(k)&= \int_{\C^3}\re\lpar \sum_{i\ne j} (\d_i\d_j
G_k)(\lambda\1_k)h_1h_2\rpar
R_{m,n,3}(\lambda,\lambda+h_1,\lambda+h_2)~\dA_3(\lambda,h_1,h_2),\\
B_{m,n}(k)&=\re \int_{\C^2}\lpar \sum_{i=1}^k (\d_i^2
G_k)(\lambda\1_k)\rpar h_1^2
R_{m,n,2}(\lambda,\lambda+h_1)~\dA_2(\lambda,h_1),\\
C_{m,n}(k)&=2\int_{\C^3}\re\lpar Z_k(\lambda) h_1\bar{h}_2\rpar
R_{m,n,3}(\lambda,\lambda+h_1,\lambda+h_2)~\dA_3(\lambda,h_1,h_2),\\
D_{m,n}(k)&=\int_{\C^2}(\Delta_k G_k)(\lambda\1_k)\babs{h_1}^2
R_{m,n,2}(\lambda,\lambda+h_1)~\dA_2(\lambda,h_1),\quad \text{and},\\
E_{m,n}(k)&=\int_{\C^{k+1}}r(\lambda,h)R_{m,n,k+1}(\lambda,\lambda\1_k+h)~\dA_{k+1}(\lambda,h).\\
\end{split}
\end{equation*}
\end{defn}

Our preceding efforts in this section are then summed up by the
following formula.

\begin{lem} \label{churm}For all $m$, $n$, $k$ and all $g\in \coity(\C)$ we have
\begin{equation}\label{terms}\calC_{m,n,k}(g)=A_{m,n}(k)+B_{m,n}(k)+C_{m,n}(k)+D_{m,n}(k)+E_{m,n}(k).
\end{equation}
\end{lem}

The rest of this paper will be devoted to a proof the following
theorem.

\begin{thm}\label{mth} Suppose that $g\in \coity(\setS_\tau^\circ\cap X)$. Then for all
$k\ge 2$ the numbers $A_{m,n}(k)$, $B_{m,n}(k)$, and $E_{m,n}(k)$
converge to $0$ as $m\to \infty$ and $n-m\tau\to 0$. Moreover we
have that
\begin{equation*}\lim_{m\to\infty,n-m\tau\to
0}D_{m,n}(k)=\begin{cases}\frac 1 2 \int_\C\babs{\nabla
g(\lambda)}^2\dA(\lambda)& \text{if}\quad k=2,\cr 0 & \text{if}\quad
k\ge 3,\end{cases}
\end{equation*}
and
\begin{equation*}\lim_{m\to\infty,n-m\tau\to
0}C_{m,n}(k)=\begin{cases}-\frac 1 4 \int_\C\babs{\nabla
g(\lambda)}^2\dA(\lambda)& \text{if}\quad k=2,\cr 0 & \text{if}\quad
k\ge 3.\end{cases}
\end{equation*}
\end{thm}

\medskip

It should be noted that Th.~\ref{mth} implies Th.~\ref{mthm}.
(Convergence of the cumulants of $\fluct_n g$ to the cumulants of
$N\lpar e_g,v_g^2\rpar$ is equivalent to convergence of the moments
which implies convergence in distribution.)

\medskip

In order to verify Th.~\ref{mth}, we will first need to look more
closely at the behaviour of the function $(\lambda,h)\mapsto
G_k(\lambda\1_k+h)R_{m,n,k+1}(\lambda,\lambda\1_k+h)$ in the next
section. We shall see that this function becomes negligible when $h$
is "large" in the sense that $\babs{h_i}\ge M_k\log m/\sqrt{m}$ for
some $i$, where $M_k$ is a sufficiently large number independent of
$m$ and $n$ as long as $\supp g\subset \setS_\tau^\circ\cap X$ and
$\babs{n-m\tau}\le 1$. This will imply that we can approximate the
integrals defining the numbers $A_{m,n}(k)$,\ldots, $E_{m,n}(k)$ by
integrals over a small neighbourhood of the diagonal in $\C^{k+1}$.

\section{Off-diagonal damping} \label{poh}

Fix a number $k\ge 2$. Throughout this section, it will be
convenient to denote
\begin{equation*}\lambda_0=\lambda_{k+1}=\lambda,\end{equation*}
so that we can write
\begin{equation*}R_{m,n,k+1}(\lambda,\ldots,\lambda_k)=\prod_{i=0}^k
K_{m,n}(\lambda_i,\lambda_{i+1})~\e^{-m(Q(\lambda_i)+Q(\lambda_{i+1}))/2}.\end{equation*}
We will frequently without further mention apply this convention in
the sequel. We will need two lemmas.

\begin{lem}\label{sgu} (\cite{B}) There is a number $C$ such that for all $z,w\in\C$ and
all $m,n$ with $n\le m\tau+1$ holds:
\begin{equation*}\babs{K_{m,n}(z,w)}^2~\e^{-m(Q(z)+Q(w))}\le
Cm^2~\e^{-m(Q(z)-\widehat{Q}_\tau(z))}~\e^{-m(Q(w)-\widehat{Q}_\tau(w))}.\end{equation*}
\end{lem}

\begin{proof} See \cite{B} or \cite{AH}, Prop. 3.6.
\end{proof}

\begin{lem}\label{nex} (\cite{AH}) Let $K$ be a compact subset of $\setS_\tau^\circ\cap
X$ and $d=\dist\,(K;\C\setminus(\setS_\tau\cap X))$. There then
exist positive numbers $C$ and $\epsilon$ depending only on $d$ such
that for all $z\in K$, $h\in\C$ and all $m,n\ge 1$ such that
$\babs{n-m\tau}\le 1$ holds:
\begin{equation*}\babs{K_{m,n}(z,z+h)}\e^{-m(Q(z)+Q(z+h))/2}\le
Cm\e^{-\epsilon\sqrt{m}\min\{d,\babs{h}\}}.\end{equation*}
\end{lem}

\begin{proof} See \cite{AH}, Th.~8.3, cf. also \cite{B2}.\end{proof}

It follows from Lemma \ref{sgu} that
\begin{equation}\label{bug}
\babs{R_{m,n,k+1}(\lambda,\lambda_1,\ldots,\lambda_k)}\le
Cm^{k+1}~\e^{-m(Q(\lambda)-\widehat{Q}_\tau(\lambda))}
~\e^{-m(Q(\lambda_1)-\widehat{Q}_\tau(\lambda_1))}
~\cdots~
\e^{-m(Q(\lambda_k)-\widehat{Q}_\tau(\lambda_k))},\end{equation}
when $n\le m\tau+1$. By the growth assumption \eqref{gro}, using
that $\tau<\rho$ and eq. \eqref{beq}, we conclude that there exists
positive numbers $C$, $C^\prime$  and $\delta$ such that
\begin{equation}\label{bab}\babs{R_{m,n,k+1}}\le C^\prime m^{k+1}\lpar
\max\{\babs{\lambda}^2,\ldots,\babs{\lambda_k}^2\}\rpar^{-m\delta}\quad
\text{when}\quad n\le m\tau+1\quad\text{and}\quad
\max\left\{\babs{\lambda}^2,\ldots,\babs{\lambda_k}^2\right\}\ge
C.\end{equation} Thus if $D_C(0)$ denotes the polydisc
$\left\{(\lambda,\ldots,\lambda_k);~\max\{\babs{\lambda}^2,
\ldots,\babs{\lambda_k}^2\}\le C\right\}$, we have for any $N\in\R$
\begin{equation*}\int_{\C^{k+1}\setminus D_C(0)} \lpar \babs{\lambda}^2+\ldots+\babs{\lambda_k}^2\rpar^N~
\babs{~R_{m,n,k+1}(\lambda,\ldots,\lambda_k)~}~
\dA_{k+1}(\lambda,\ldots,\lambda_k)\to 0,\quad \text{as}\quad
m\to\infty,\, n\le m\tau+1,\end{equation*} when $C$ is large enough.
We shall now show that much more is true. We first have the
following lemma. In the proofs we conform to previous notation and
write
\begin{equation*}\delta_m=\log m/\sqrt{m}.\end{equation*}
We also put
\begin{equation*}d=\dist\,\lpar\supp g;\C\setminus(\setS_\tau\cap
X)\rpar,\end{equation*} and
\begin{equation}\label{ck}K=\left\{z\in \C;\dist\,(z;\C\setminus(\setS_\tau\cap X))\ge
d/2\right\}.\end{equation}
 We also remind the reader of the convention
that $\lambda_{k+1}=\lambda_0=\lambda$.

\begin{lem} \label{grund} There exists positive numbers $M$, $\alpha$ and $m_0$ depending only on $k$ and $d$
such that if $\lambda_j\in K$ and $\babs{\lambda_j-\lambda_{j+1}}\ge
M\delta_m$ for some index $j\in\{0,\ldots,k\}$, then for all $m\ge
m_0$
\begin{equation*}\babs{~R_{m,n,k+1}
(\lambda_0,\lambda_1,\ldots,\lambda_k)~}\le
Cm^{-\alpha},\quad \babs{n-m\tau}\le 1,\end{equation*} where $C$
depends only on $d$.
\end{lem}

\begin{proof} In view of Lemma \ref{nex}, the hypothesis yields that
\begin{equation*}\babs{~K_{m,n}(\lambda_j,\lambda_{j+1})~}~\e^{-m(Q(\lambda_j)+Q(\lambda_{j+1}))/2}
\le
Cm~\e^{-\epsilon\sqrt{m}\min\left\{d/2,
\babs{\lambda_j-\lambda_{j+1}}\right\}},\quad
\babs{n- m\tau}\le 1,\end{equation*} with numbers $C$ and $\epsilon$
depending only on $d$, and $\babs{\lambda_j-\lambda_{j+1}}\ge
M\delta_m$. Choosing $m_0$ large enough that $M\delta_m\le d/2$ for
$m\ge m_0$ it yields that
\begin{equation}\label{dos}
\babs{~K_{m,n}(\lambda_j,\lambda_{j+1})~}~\e^{-m(Q(\lambda_j)+Q(\lambda_{j+1}))/2}
\le Cm~\e^{-\epsilon\sqrt{m}M\delta_m}=Cm^{1-\epsilon M},\quad
\babs{n- m\tau}\le 1,\end{equation} when $m\ge m_0$. On the other
hand, if $n\le m\tau+1$, Lemma \ref{sgu} yields that
\begin{equation}\label{tres}
\babs{~K_{m,n}(\lambda_l,\lambda_{l+1})~}~\e^{-m(Q(\lambda_l)+Q(\lambda_{l+1}))/2}\le
Cm,\quad l=0,\ldots,k.\end{equation} Now \eqref{dos} and
\eqref{tres} implies
\begin{equation}\label{fore}
\babs{~R_{m,n,k+1}(\lambda_0,\ldots,\lambda_k)~}=\prod_{l=0}^k
\babs{~K_{m,n}(\lambda_l,\lambda_{l+1})~}~\e^{-m(Q(\lambda_l)+Q(\lambda_{l+1}))/2}\le
Cm^{k+1-\epsilon M},\end{equation} when $m\ge m_0$ and
$\babs{n-m\tau}\le 1$. It now suffices to choose $M$ large enough
that
\begin{equation*}\epsilon M-k-1>0,\end{equation*}
and then put $\alpha=\epsilon M-k-1$.
\end{proof}

\medskip

We henceforth let $M$ denote a fixed large number with the
properties provided by Lemma \ref{grund}. Let us also put
\begin{equation*}U_g(\lambda)=\dist\,\lpar\lambda;\supp
g\rpar,\quad \lambda\in\C,\end{equation*}
\begin{equation*}U_g^*(\lambda_0,\ldots,\lambda_k)=\max\left\{U_g(\lambda_i);i=0,\ldots,k\right\},\end{equation*}
and
\begin{equation*}V_{m,k}=\biggl\{U_g^*(\lambda_0,\ldots,\lambda_k)\ge Mk\delta_m\biggr\}.\end{equation*}

\begin{lem} \label{vlem} The function
\begin{equation}\label{fun}(\lambda_0,\lambda_1,\ldots,\lambda_k)\mapsto
G_k(\lambda_1,\ldots,\lambda_k)~R_{m,n,k+1}(\lambda_0,\lambda_1,\ldots,\lambda_k),\end{equation}
converges to zero uniformly on the set $V_{m,k}$ as $m\to\infty$ and
$\babs{n-m\tau}\le 1$.
\end{lem}

\begin{proof} Since $G_k$ is bounded, it suffices to prove that
$R_{m,n,k+1}$ converges to zero uniformly on the set
\begin{equation*}V_{m,k}^\prime=V_{m,k}\cap \supp G_k.\end{equation*}
Here we regard $G_k$ as a function of the variables
$\lambda_0,\ldots,\lambda_k$, which is independent of the parameter
$\lambda_0$. It is then clear that
\begin{equation*}\supp G_k\subset\left\{(\lambda_0,\ldots,\lambda_k);\,
\lambda_0\in\C,\,\text{and}\, \lambda_i\in\supp g\, \text{for some
}i=1,\ldots,k\right\}.\end{equation*} Thus if
$(\lambda_0,\ldots,\lambda_k)\in V_{m,k}^\prime$, then there exists
an index $i\in\{1,\ldots,k\}$ such that $\lambda_i\in \supp g$.
Since the function $R_{m,n,k+1}(\lambda_0,\ldots,\lambda_k)$ is
invariant under the cyclic permutation $0\mapsto
1\mapsto\ldots\mapsto k\mapsto 0$ of the indices, we can
\textit{w.l.o.g.} assume that $i=1$. Then, since $U_g(\lambda_1)=0$
and $U_g^*(\lambda_1,\ldots,\lambda_{k+1})\ge Mk\delta_m$, there
must exist an integer $j\in\{1,\ldots,k\}$ such that
$\babs{\lambda_l-\lambda_{l+1}}< M\delta_m$ for all indices $l$ with
$1\le l<j$ and $\babs{\lambda_j-\lambda_{j+1}}\ge M\delta_m$. It
then follows from the triangle inequality that
\begin{equation}\label{bou}U_g(\lambda_j)\le\babs{\lambda_j-\lambda_1}<Mk\delta_m.\end{equation} If $m$ is
large enough that \begin{equation}\label{lm}Mk\delta_m\le
d/2,\end{equation} then \eqref{bou} implies that $\lambda_j$ belongs
to the compact set $K$ (see \eqref{ck}) and
$\babs{\lambda_j-\lambda_{j+1}}\ge M\delta_m$. Hence Lemma
\ref{grund} yields that
\begin{equation*}\babs{~R_{m,n,k+1}(\lambda_0,\ldots,\lambda_k)~}\le
Cm^{-\alpha}\end{equation*} for large $m$ when $\babs{n-m\tau}\le
1$, where $\alpha>0$. This proves that $R_{m,n,k+1}$ converges
uniformly to $0$ on $V_{m,k}^\prime$.
\end{proof}

\medskip

Let us now put
\begin{equation*}N(\lambda_0,\ldots,\lambda_k)=\max_{0\le i\le
k}\biggl\{\babs{\lambda_i-\lambda_{i+1}}\biggr\}.\end{equation*} We
shall next prove that the function $G_kR_{m,n,k+1}$ is uniformly
small on the set
\begin{equation*}W_{m,k}:=\left\{(\lambda_0,\ldots,\lambda_k);\, U_g^*(\lambda_0,\ldots,\lambda_k)\ge
Mk\delta_m\quad \text{or}\quad N(\lambda_0,\ldots,\lambda_k)\ge
M\delta_m\right\},\end{equation*} where $M=M(k,d)$ is a number provided by
Lemma \ref{vlem}.

\begin{lem} \label{wllem}The function
\begin{equation}\label{funk}(\lambda_0,\lambda_1,\ldots,\lambda_k)\mapsto
G_k(\lambda_1,\ldots,\lambda_k)~R_{m,n,k+1}(\lambda_0,\lambda_1,\ldots,\lambda_k)\end{equation}
converges to zero uniformly on $W_{m,k}$ as $m\to\infty$ and
$\babs{n-m\tau}\le 1$.
\end{lem}

\begin{proof} By Lemma \ref{vlem} we know that the function
\eqref{funk} converges to zero uniformly on the set $\{U_g^*\ge
Mk\delta_m\}$. It thus suffices to show uniform convergence on the
set
\begin{equation*}W_{m,k}^\prime=\left\{U_g^*(\lambda_0,\ldots,\lambda_k)\le
Mk\delta_m\quad\text{and}\quad N(\lambda_0,\ldots,\lambda_k)\ge
M\delta_m\right\}.\end{equation*} Now note that if $m$ is large enough
that $Mk\delta_m\le d/2$, we will have
\begin{equation*}W_{m,k}^\prime\subset K,\end{equation*}
with $K$ as in \eqref{ck}. Hence if $(\lambda_0,\ldots,\lambda_k)\in
W_{m,k}^\prime$, we will have that $\lambda_i\in K$ and
$\babs{\lambda_i-\lambda_{i+1}}\ge M\delta_m$ for some $i$. It then
follows from Lemma \ref{grund} that
$\babs{~R_{m,n,k+1}(\lambda_0,\ldots,\lambda_k)~}\le Cm^{-\alpha}$
when $\babs{n-m\tau}\le 1$, where $\alpha>0$. It follows that
$R_{m,n,k+1}\to 0$ uniformly on $W_{m,k}^\prime$, and the lemma
follows.\end{proof}

\medskip

It is now advantageous to pass to the coordinate system
$(\lambda,h)$ where $\lambda=\lambda_0$ and $h_i=\lambda_i-\lambda$
for $i=1,\ldots, k$. Let us put
\begin{equation*}\babs{h}_\infty=\max\{\babs{h_i};1\le i\le
k\},\end{equation*} and
\begin{equation}\label{ydef}Y_{m,k}=\left\{(\lambda,h)\in\C^{k+1};\, U_g(\lambda)\le Mk\delta_m,\,
\babs{h}_\infty\le Mk\delta_m\right\}.\end{equation} As we shall see,
everything interesting goes on in the set $Y_{m,k}$ when $m$ is
large and $\babs{n-m\tau}\le 1$.

\begin{lem} \label{fa} The function
\begin{equation*}(\lambda,h)\mapsto
G_k(\lambda\1_k+h)~R_{m,n,k+1}(\lambda,\lambda\1_k+h)\end{equation*}
converges to zero uniformly on the complement of $Y_{m,k}$ as
$m\to\infty$ and $\babs{n-m\tau}\le 1$.
\end{lem}

\begin{proof} In view of Lemma \ref{wllem}, it
suffices to prove that if $(\lambda,h)$ is in the complement of
$Y_{m,k}$, then $(\lambda,\lambda_1,\ldots,\lambda_k)$ belongs to
$W_{m,k}$, where $\lambda_i=\lambda+h_i$. But if $(\lambda,h)\not\in
Y_{m,k}$, then either $U_g(\lambda)> Mk\delta_m$, or
$\babs{\lambda-\lambda_i}> Mk\delta_m$ for some $i=1,\ldots,k$. But
the latter inequality can only hold if
$\babs{\lambda_j-\lambda_{j+1}}>M\delta_m$ for some $j$, whence
$N(\lambda,\lambda_1,\ldots,\lambda_k)\ge M\delta_m$. Thus, in
either case, we have $(\lambda_0,\ldots,\lambda_k)\in W_{m,k}$ and
the lemma follows.
\end{proof}

\medskip

The following result sums up our efforts in this section, and is what
is needed to prove the asymptotic behaviour of the cumulants in the
next section.

\begin{lem}\label{gurka}
We have that
\begin{equation*}\int_{\C^{k+1}\setminus Y_{m,k}}
\babs{~G_k(\lambda\1_k+h)~
R_{m,n,k+1}(\lambda,\lambda\1_k+h))~}~\dA_{k+1}(\lambda,h)\to
0,\end{equation*} as $m\to\infty$ and $\babs{n-m\tau}\le 1$.
\end{lem}

\begin{proof} It follows from \eqref{bab} that the
integrals
\begin{equation*}I_m=\int_{\C^{k+1}\setminus
D_C(0)}~G_k(\lambda\1_k+h)~R_{m,n,k+1}(\lambda,\lambda\1_k+h)~\dA_{k+1}(\lambda,h)\end{equation*}
converge absolutely for large enough $m$ and $C$ if $n\le m\tau+1$,
and $I_m\to 0$ as $m\to \infty$ and $n\le m\tau+1$. The statement
now follows from Lemma \ref{fa}.
\end{proof}

\begin{rem}\label{throw} Suppose that $P(\lambda,h)$ is a measurable function on $\C^{k+1}$ such that (i) $P(\lambda,h)\equiv 0$ when $\lambda\not\in \supp g$
and (ii) $\babs{P(\lambda,h)}\le C\lpar 1+\babs{h}^2\rpar^N$ for some constants $C$ and $N$. (We write $\babs{h}$ for the $\ell^2$ norm on $\C^k$, so that $\babs{h}_\infty^2\le \babs{h}^2\le k\babs{h}_\infty^2$.)

As above, we can then conclude that
\begin{equation}\label{sista}\int_{\C^{k+1}\setminus Y_{m,k}}P(\lambda,h)~R_{m,n,k+1}(\lambda,\lambda\1_k+h)~\dA_{k+1}(\lambda,h)\to 0,\quad\text{as}\quad m\to\infty,\, \babs{n-m\tau}\le 1.\end{equation}
Indeed, \eqref{sista} follows from Lemma \ref{grund}, if we use also the estimate \eqref{ops} to estimate the part of integral over $\babs{h}\ge C$ for $C$ large enough.
The details of a proof parallel our proof of Lemma \ref{gurka}, but are simpler in the present case, since $U_g(\lambda)=0$ when $P\ne 0$.
\end{rem}

\section{Conclusion of the proof of Theorem \ref{mth}} \label{puuh}

In this section, we prove Th.~\ref{mth}. As we have observed earlier, this
theorem implies Th.~\ref{mthm}, and thus the story ends with this
section.

Our proof will be accomplished by estimating the various terms in
the identity
\begin{equation*}\calC_{m,n,k}(g)=A_{m,n}(k)+B_{m,n}(k)+C_{m,n}(k)+D_{m,n}(k)+E_{m,n}(k),\end{equation*}
see \eqref{terms}. We start by considering the ``error-term''
\begin{equation*}E_{m,n}(k)=\int_{\C^{k+1}}r(\lambda,h)~R_{m,n,k+1}(\lambda,\lambda\1_k+h)~\dA_{k+1}(\lambda,h),
\end{equation*}
where $r(\lambda,h)$ is the remainder term of order 3 from Taylor's
formula applied to the function $h\mapsto G_k(\lambda\1_k+h)$ at
$h=0$, see \eqref{t2}. We have that $r(\lambda,h)=G_k(\lambda\1_k+h)-P_2(\lambda,h)$ where $P_2$ is a polynomial of degree $2$ in $h$ with the property
that $P_2(\lambda,h)=0$ when $\lambda\not\in\supp(g)$. It follows from Remark \ref{throw} that, when $m\to\infty$ and $\babs{n-m\tau}\le 1$,  (with $Y_{m,k}$ as
in \eqref{ydef})
\begin{equation}\label{ops}\int_{\C^{k+1}\setminus Y_{m,k}}P_2(\lambda,h)~R_{m,n,k+1}(\lambda,\lambda\1_k+h)~\dA_{k+1}(\lambda,h)\to 0.\end{equation}
Using \eqref{ops} and Lemma \ref{gurka} we conclude that
\begin{equation*}\int_{\C^{k+1}\setminus Y_{m,k}}
r(\lambda,h)~R_{m,n,k+1}(\lambda,\lambda\1_k+h)~\dA_{k+1}(\lambda,h)\to
0,\end{equation*} when $m\to\infty$ and $\babs{n-m\tau}\le 1$. In
order to estimate the integral over $Y_{m,k}$, we first introduce
some notation.

For a measurable subset $\Omega\subset \C^{N}$, let us denote the
(suitably normalized) complex $N$-dimensional volume of $U$ by
$\Vol_{N}(\Omega)=\int_{\Omega}\dA_N(\lambda_1,\ldots, \lambda_N)$.
When $N=1$ we write $\Area(\Omega)$ in stead of $\Vol_1(\Omega)$.

For large $m$, the set $Y_{m,k}$ is contained in the set
\begin{equation*}\left\{(\lambda,h);
\lambda\in \setS_\tau, \babs{h}_\infty\le
Mk\delta_m\right\},\end{equation*} whence
\begin{equation*}\Vol_{k+1}(Y_{m,k})\le \Area(\setS_\tau)~(Mk\delta_m)^{2k}=
C\delta_m^{2k},\end{equation*} with $C$ a number depending on $k$,
$M$ and $\tau$. Furthermore, \eqref{bug} yields that
\begin{equation*}\babs{~R_{m,n,k+1}(\lambda,\lambda \1_k+h))~}\le
Cm^{k+1},\quad n\le m\tau+1,\end{equation*} for all $\lambda$ and
$h$. Now, since $\babs{r(\lambda,h)}\le C\babs{h}^3\le C\delta_m^3$
when $\babs{h}\le Mk\delta_m$, it yields
\begin{equation*}\begin{split}\int_{Y_{m,k}}\babs{~r(\lambda,h)~R_{m,n,k+1}(\lambda,\lambda\1_k+h)}~\dA_{k+1}(\lambda,h)
&\le C\delta_m^3~m^{k+1}~\Vol_{k+1}(Y_{m,k})=\\
&=Cm^{k+1}~\delta_m^{2k+3}=C\log^{2k+3} m/\sqrt{m}.\\
\end{split}\end{equation*}
Hence also the integral over $Y_{m,k}$ converges to $0$ when
$m\to\infty$ and $\babs{n-m\tau}\le 1$. We have shown that
$E_{m,n}(k)\to 0$ as $m\to\infty$ and $\babs{n-m\tau}\le 1$.

We next consider the term
\begin{equation*}D_{m,n}(k)=\int_{\C^2}(\Delta_k
G_k)(\lambda\1_k)~\babs{h_1}^2~
R_{m,n,2}(\lambda,\lambda+h_1)~\dA_2(\lambda,h_1).\end{equation*} In
view of Lemma \ref{zerolap}, we plainly have \begin{equation*}
D_{m,n}(k)=0\quad \text{if}\quad k\ge 3.\end{equation*} It thus
remains to consider the case $k=2$. In this case, Lemma
\ref{zerolap} implies
\begin{equation*}D_{m,n}(2)=\frac 1 2 \int_{\C^2}\babs{\nabla
g(\lambda)}^2~\babs{h}^2~
R_{m,n,2}(\lambda,\lambda+h)~\dA_2(\lambda,h).\end{equation*} It is
clear from Remark \ref{throw} that
\begin{equation}\label{use1}\int_{\babs{h}\ge 2M\delta_m}\babs{\nabla
g(\lambda)}^2
\babs{h}^2R_{m,n,2}(\lambda,\lambda+h)~\dA_2(\lambda,h)\to
0,\end{equation} as $m\to\infty$ and $\babs{n-m\tau}\le 1$. To
estimate the integral over $\{\babs{h}\le 2M\delta_m\}$ we apply the
asymptotics for $R_{m,n,2}$ from eq. \eqref{apu} (with the compact
set $K$ replaced by $\supp g$). It yields that there are numbers
$v_m$ converging to $1$ when $m\to\infty$ such that
\begin{equation}\label{use2}\begin{split}\int_{\babs{h}\le 2M\delta_m}&\babs{\nabla
g(\lambda)}^2 \babs{h}^2R_{m,n,2}(\lambda,\lambda+h)~\dA_2(\lambda,h)=\\
&=v_m m^2\int_{\babs{h}\le 2M\delta_m}\babs{\nabla
g(\lambda)}^2\babs{h}^2\lpar \Delta
Q(\lambda)^2+\calO(\delta_m)\rpar~\e^{-m\Delta
Q(\lambda)\babs{h}^2}\dA_2(\lambda,h)+o(1),\\
\end{split}
\end{equation}
when $m\to\infty$ and $n\ge m\tau-1$. Now, for a fixed
$\lambda\in\supp g$, the change of variables $\xi=\sqrt{m\Delta
Q(\lambda)}h$ shows that
\begin{equation*}\begin{split}
\int_{\babs{h}\le 2M\delta_m} \lpar m\Delta
Q(\lambda)\rpar^2\babs{h}^2\e^{-m\Delta Q(\lambda)\babs{h}^2}\dA(h)
=\int_{\babs{\xi}\le 2M\log m}\babs{\xi}^2
\e^{-\babs{\xi}^2}\dA(\xi)\to 1,\\
\end{split}\end{equation*}
as $m\to\infty$. Hence it follows from \eqref{use1} and \eqref{use2}
that
\begin{equation*}D_{m,n}(2)\to \frac 1 2 \int_\C\babs{\nabla
g(\lambda)}^2\dA(\lambda),\end{equation*} as $m\to \infty$ and
$\babs{n-m\tau}\le 1$.

The complete asymptotics for $D_{m,n}(k)$ has now been settled, and
we turn to the term
\begin{equation*}B_{m,n}(k)=\re\int_{\C^2}S(\lambda)~ h^2~
R_{m,n,2}(\lambda,\lambda+h)~\dA_2(\lambda,h),\end{equation*} where
we have put \begin{equation*}S(\lambda)=\sum_{i=1}^k (\d_i^2
G_k)(\lambda\1_k).\end{equation*} Note that $\supp S\subset \supp
g$. Using Remark \ref{throw}, we obtain (as before) that
\begin{equation*}\int_{\babs{h}\ge 2M\delta_m}S(\lambda)~h^2~
R_{m,n,2}(\lambda,\lambda+h)~\dA_2(\lambda,h)\to 0,\end{equation*} as
$m\to \infty$ and $\babs{n-m\tau}\le 1$. When $\babs{h}\le
2M\delta_m$ we again use the asymptotics in \eqref{apu}, which
yields that there are numbers $v_m$ converging to $1$ as $m\to
\infty$ such that
\begin{equation}\label{use3}\begin{split}&
\int_{\babs{h}\le 2M\delta_m}S(\lambda)~
h^2~R_{m,n,2}(\lambda,\lambda+h)~\dA_2(\lambda,h)=\\
&=v_m m^2\int_{\babs{h}\le 2M\delta_m}S(\lambda)~h^2\lpar\Delta
Q(\lambda)^2+\calO(\delta_m)\rpar~\e^{-m\Delta
Q(\lambda)\babs{h}^2}~\dA_2(\lambda,h)+o(1).\\
\end{split}
\end{equation}
Now, using that, for a fixed $\lambda\in\supp g$,
\begin{equation*}\begin{split}&\int_{\babs{h}\le 2M\delta_m}(m\Delta
Q(\lambda))^2h^2\e^{-m\Delta Q(\lambda)\babs{h}^2}\dA(h)
=\int_{\babs{\xi}\le 2M\log m}\xi^2\e^{-\babs{\xi}^2}\dA(\xi)=0,\\
\end{split}\end{equation*}
 we infer that $B_{m,n}(k)\to 0$ for all $k\ge 2$ as $m\to \infty$
and $\babs{n-m\tau}\le 1$.

There remains to estimate the terms $A_{m,n}(k)$ and $C_{m,n}(k)$.
These terms are a little more complicated than the previous ones
since they are defined as integrals over $\C^3$ and not over $\C^2$.
We first turn to the term $A_{m,n}(k)$ which we now write in the
form
\begin{equation*}A_{m,n}(k)=\frac 1 2 \int_{\C^3}\lpar
T(\lambda)h_1h_2+\overline{T(\lambda)}\bar{h}_1\bar{h}_2\rpar
R_{m,n,3}(\lambda,\lambda+h_1,\lambda+h_2)~\dA_3(\lambda,h_1,h_2),\end{equation*}
where we have put
\begin{equation*}T(\lambda)=\sum_{i\ne j} (\d_i\d_j
G_k)(\lambda\1_k).\end{equation*} It is clear that $\supp
T\subset\supp g$. Furthermore, using Remark \ref{throw}, we see as
before that, with $h=(h_1,h_2)$ and
$\babs{h}_\infty=\max\{\babs{h_1},\babs{h_2}\}$,
\begin{equation*}\int_{\babs{h}_\infty\ge 3M\delta_m}
\re\lpar T(\lambda)h_1h_2\rpar
R_{m,n,3}(\lambda,\lambda+h_1,\lambda+h_2)~\dA_3(\lambda,h_1,h_2)\to
0,\end{equation*} as $m\to\infty$ and $\babs{n-m\tau}\le 1$. When
$\babs{h}_\infty\le 3M\delta_m$, insert the asymptotics for
$R_{m,n,3}$ provided by eq. \eqref{apuu}. It shows that there are
numbers $v_m$ converging  to $1$ as $m\to \infty$ such that
\begin{equation*}\begin{split}&\int_{\babs{h}_\infty\le 3M\delta_m}
T(\lambda)~h_1~h_2~R_{m,n,3}(\lambda,\lambda+h_1,\lambda+h_2)~\dA_3(\lambda,h)=\\
&=m^3v_m\int_{\babs{h}_\infty\le 3M\delta_m}T(\lambda)~h_1~h_2~\lpar\Delta
Q(\lambda)^3+\calO(\delta_m)\rpar~\e^{m\Delta
Q(\lambda)(h_1\bar{h}_2-\babs{h_1}^2-\babs{h_2}^2)}~\dA_3(\lambda,h)+
o(1).\\
\end{split}
\end{equation*}
Now fix $\lambda\in\supp g$ and put $\xi_1=\sqrt{m\Delta
Q(\lambda)}h_1$ and $\xi_2=\sqrt{m\Delta Q(\lambda)} h_2$. We then
have that
\begin{equation*}\begin{split}&m^3v_m
\int_{\babs{h}_\infty\le 3M\delta_m}T(\lambda)\lpar \Delta
Q(\lambda)^3+\calO(\delta_m)\rpar h_1~h_2~\e^{m\Delta
Q(\lambda)(h_1\bar{h}_2-\babs{h_1}^2-\babs{h_2}^2)}~\dA_2(h)=\\
&=T(\lambda)\int_{\babs{\xi}_\infty\le 3M\log m}\lpar 1+\calO(\delta_m)\rpar
\xi_1~\xi_2~\e^{\xi_1\bar{\xi}_2-\babs{\xi_1}^2-\babs{\xi_2}^2}~\dA_2(\xi).\\
\end{split}
\end{equation*}
Thus when we can prove that $J=0$ and $J'=0$ where
\begin{equation}\label{sponk}J=\int_{\C^2}\xi_1~\xi_2~\e^{\xi_1\bar{\xi}_2-\babs{\xi_1}^2-\babs{\xi_2}^2}
~\dA_2(\xi_1,\xi_2)\quad \text{and}\quad
J'=\int_{\C^2}\bar{\xi}_1~\bar{\xi}_2~\e^{\xi_1\bar{\xi}_2-\babs{\xi_1}^2-\babs{\xi_2}^2}~
\dA_2(\xi_1,\xi_2)
\end{equation}
we will obtain the result that $A_{m,n}(k)\to 0$ as $m\to\infty$ and
$\babs{n-m\tau}\le 1$ for all $k\ge 2$.

The argument for $J'$ is similar so we settle for proving that
$J=0$. To this end, we write the integral in polar coordinates:
\begin{equation*}J=\frac 1 {\pi^2}\int_0^\infty\int_0^\infty I(r,\rho)~\diff r~\diff
\rho,\end{equation*} where
\begin{equation*}I(r,\rho)=\int_0^{2\pi}\int_0^{2\pi}
(r\rho)^2~\e^{\imag(\theta+\phi)}~\e^{r\rho\e^{\imag(\theta-\phi)}-r^2-\rho^2}~\diff\phi~\diff
\theta.\end{equation*} Performing the change of variables
$\vt=\theta+\pi/2$ and $\vf=\phi+\pi/2$, the latter integral
transforms to
\begin{equation*}I(r,\rho)=\int_0^{2\pi}\int_0^{2\pi}(r\rho)^2~
\e^{\imag(\pi+\vt+\vf)}~\e^{r\rho\e^{\imag(\vt-\vf)}-r^2-\rho^2}~\diff\vt~\diff\vf=-I(r,\rho).\end{equation*}
Hence $I(r,\rho)=0$ for all $r$ and $\rho$ and it follows that
$J=0$.

There remains to consider the term
\begin{equation*}C_{m,n}(k)= \int_{\C^3}\lpar
Z_k(\lambda)h_1\bar{h}_2+\overline{Z_k(\lambda)} \bar{h}_1h_2\rpar~
R_{m,n,3}(\lambda,\lambda+h_1,\lambda+h_2)~\dA_3(\lambda,h_1,h_2),\end{equation*}
where
\begin{equation*}Z_k(\lambda)=\sum_{i< j}(\d_i\dbar_j
G_k)(\lambda\1_k).\end{equation*} Observing that $\supp
Z_k\subset\supp g$ and arguing is in the case of $A_{m,n}(k)$, it is
seen that
\begin{equation*}\int_{\babs{h}_\infty\ge 3M\delta_m}
Z_k(\lambda)~h_1~\bar{h}_2~R_{m,n,3}(\lambda,\lambda+h_1,\lambda+h_2)~
\dA_3(\lambda,h_1,h_2)\to 0,
\end{equation*}
as $m\to\infty$ and $\babs{n-m\tau}\le 1$. Hence, using
\eqref{apuu}, we obtain that the asymptotics of $C_{m,n}(k)$ is that
of $C_{m,n}^\prime(k)+C_{m,n}^{\prime\prime}(k)$ where
\begin{equation*}\begin{split}C_{m,n}^\prime(k)&=\int_{\babs{h}_\infty\le 3M\delta_m}
Z_k(\lambda)~h_1~\bar{h}_2~R_{m,n,3}(\lambda,\lambda+h_1,\lambda+h_2)~
\dA_3(\lambda,h_1,h_2)=\\
&=m^3 v_m\int_{\babs{h}_\infty\le 3M\delta_m}
Z_k(\lambda)~h_1~\bar{h}_2~ \lpar \Delta Q(\lambda)^3+\calO(\delta_m)\rpar~
\e^{m\Delta Q(\lambda)(h_1\bar{h}_2-\babs{h_1}^2-\babs{h_2}^2)}~
\dA_3(\lambda,h)=\\
&=v_m\int_\C Z_k(\lambda)\lpar \int_{\babs{\xi}_\infty\le 3M\log m}
\lpar 1+\calO(\delta_m)\rpar~\xi_1~\bar{\xi}_2~
\e^{\xi_1\bar{\xi}_2-\babs{\xi_1}^2-\babs{\xi_2}^2}~\dA_2(\xi_1,\xi_2)\rpar~
\dA(\lambda),\\
\end{split}\end{equation*}
and (likewise)
\begin{equation}\label{ppeq}C_{m,n}^{\prime\prime}(k)=v_m\int_\C \overline{Z_k(\lambda)}~\lpar \int_{\babs{\xi}_\infty\le 3M\log m}
\lpar 1+\calO(\delta_m)\rpar ~\bar{\xi}_1~\xi_2~
\e^{\xi_1\bar{\xi}_2-\babs{\xi_1}^2-\babs{\xi_2}^2}~\dA_2(\xi_1,\xi_2)\rpar~
\dA(\lambda),\end{equation} where $v_m\to 1$ as $m\to \infty$.

We first claim that $C_{m,n}^\prime(k)\to 0$ when $m\to\infty$ and
$\babs{n-m\tau}\le 1$ for all $k\ge 2$. We will have shown that when
we can prove that $L'=0$ where
\begin{equation*}L'=\int_{\C^2}
\xi_1\bar{\xi}_2\e^{\xi_1\bar{\xi}_2-\babs{\xi_1}^2-\babs{\xi_2}^2}\dA_2(\xi_1,\xi_2).\end{equation*}
To prove this, we pass to polar coordinates and write
\begin{equation*}L'=\frac 1 {\pi^2}\int_0^\infty\int_0^\infty P(r,\rho)\diff r\diff\rho,
\end{equation*}
where
\begin{equation*}P(r,\rho)=\int_0^{2\pi}\int_0^{2\pi}
(r\rho)^2~\e^{\imag(\theta-\phi)}~\e^{r\rho\e^{i(\theta-\phi)}-r^2-\rho^2}~
\diff\theta~\diff\phi.\end{equation*} Making the change of variables
$\vt=\theta-\phi$ and $\vf=\phi$, the integral transforms to
\begin{equation*}P(r,\rho)=\e^{-r^2-\rho^2}
\int_0^{2\pi}\lpar \int_{-\vf}^{2\pi-\vf} (r\rho)^2\e^{\imag
\vt}\e^{r\rho\e^{\imag\vt}}\diff\vt\rpar\diff\vf.\end{equation*} But
the inner integral is readily calculated,
\begin{equation*}\int_{-\vf}^{2\pi-\vf}
(r\rho)^2\e^{\imag \vt}\e^{r\rho\e^{\imag\vt}}\diff\vt=\biggl[
-\imag r\rho\e^{r\rho\e^{\imag\vt}}\biggr]_{\vt=-\vf}^{2\pi-\vf}=0.
\end{equation*}
This shows that $P(r,\rho)=0$ and consequently $L'=0$. It follows
that $C_{m,n}^\prime(k)\to 0$ as $m\to\infty$ and $\babs{n-m\tau}\le
1$ for all $k\ge 2$.

To handle the term $C_{m,n}^{\prime\prime}(k)$, it becomes necessary
to calculate
\begin{equation*}L^{\prime\prime}=\int_{\C^2}
\bar{\xi}_1\xi_2\e^{\xi_1\bar{\xi}_2-\babs{\xi_1}^2-\babs{\xi_2}^2}\dA_2(\xi_1,\xi_2).\end{equation*}
Again passing to polar coordinates, we write
\begin{equation*}L^{\prime\prime}=\frac 1 {\pi^2}\int_0^\infty\int_0^\infty
W(r,\rho)~\diff r~\diff\rho,\end{equation*} where
\begin{equation*}W(r,\rho)=\e^{-r^2-\rho^2}\int_0^{2\pi}
\int_{0}^{2\pi}(r\rho)^2~\e^{\imag(\theta-\phi)}~\e^{r\rho\e^{\imag(\phi-\theta)}}~\diff\phi~\diff\theta=2\pi\e^{-r^2-\rho^2}
\int_0^{2\pi}(r\rho)^2~\e^{-i\vt}~\e^{r\rho\e^{i\vt}}~\diff\vt.\end{equation*}
We now put $z=\e^{\imag\vt}$ and use a simple residue argument to
get
\begin{equation*}W(r,\rho)=\frac {2\pi (r\rho)^2~\e^{-r^2-\rho^2}} \imag \int_{\T}\frac {1} {z^2}\e^{r\rho z}~\diff
z=4\pi^2 (r\rho)^3~\e^{-r^2-\rho^2}.\end{equation*} It follows that
\begin{equation}\label{stam}\begin{split}L^{\prime\prime}&=4\int_0^\infty\int_0^\infty
(r\rho)^3~\e^{-r^2-\rho^2}~\diff r~\diff\rho=1.\\
\end{split}
\end{equation}

For $k=2$ it now follows from \eqref{stam}, \eqref{ppeq} and Lemma
\ref{zgr} that
\begin{equation*}C_{m,n}^{\prime\prime}(2)\to -\int_\C\babs{~\dbar
g(\lambda)~}^2~\dA(\lambda),\end{equation*} when $m\to\infty$ and
$\babs{n-m\tau}\le 1$. On the other hand when $k\ge 3$ we get that
\begin{equation}\label{skurm}\lim_{m\to\infty,\, \babs{n-m\tau}\le 1}C_{m,n}^{\prime\prime}(k)=
\int_\C\overline{Z_k(\lambda)}~\dA(\lambda)\end{equation} is pure
imaginary, again by Lemma \ref{zgr}. In fact this shows that the
limit in \eqref{skurm} must vanish, because the cumulant
$\calC_{m,n,k}(g)$ is real and all other terms in the expansion (in
Lemma \ref{churm}) but $C_{m,n}(k)$ have already been shown to be
real (in fact zero) in the limit when $m\to\infty$ and
$\babs{n-m\tau}\le 1$.

The proofs of all statements are now
complete. \hfill q.e.d.

\section{Concluding remarks}\label{conclu}

We conclude this paper with a series of remarks concerning possible applications and generalizations of the main theorem.
We also outline an alternative approach to the proof of Th. \ref{mthm}.

\subsection{Non-analytic potentials} \label{nap} Recall that we proved Th.~\ref{mthm} assuming that the potential $Q$ is \textit{real-analytic}
in some neighbourhood of $\setS_\tau$. It is possible to extend
this result to more general smooth potentials. Assuming that $Q$ is
$\calC^\infty$-smooth, one defines the auxiliary functions $\psi$,
$b_0$ and $b_1$ in the expression
\begin{equation*}K_m^1(z,w)=\lpar
mb_0(z,\bar{w})+b_1(z,\bar{w})\rpar\e^{m\psi(z,\bar{w})}\end{equation*}
as any fixed almost-holomorphic extensions from the anti-diagonal of
$Q$, $\Delta Q$ and $\frac 1 2 \Delta\log\Delta Q$ respectively. For
example, in the case of $\psi$ this means that $\psi$ is
well-defined and smooth in a neighbourhood of the anti-diagonal in
$\C^2$, and (i) $\psi\lpar z,\bar{z}\rpar=Q(z)$, (ii) the anti-holomorphic derivatives
$\dbar_i\psi$ vanish to infinite order at each point of the anti-diagonal, $i=1,2$,
and (iii) $\psi(z,w)=\overline{\psi(\bar{w},\bar{z})}$ whenever the
expressions make sense. Lemma \ref{klam}
extends to this more general situation; the proof is not very different from the argument in  \cite{AH} but it involves some additional technical work. The rest of the proof of Th.~\ref{mthm} for smooth potentials  requires only minor
changes.

As  we mentioned earlier, the smoothness (or analyticity) condition is "local" -- we need it only in some neighborhood of the droplet. In particular,
Theorem \ref{mthm} is true for
 potentials $Q:\C\to \R\cup\{+\infty\}$ of the form
$$Q(z)=Q_0(z)+\int_\C\log\frac 1 {\babs{z-z_0}^2}\diff\mu(z_0),$$ where $Q_0$ is a smooth function (with sufficient growth at infinity), and $\mu$ is a positive,
finitely supported measure (linear combination of
Dirac measures). In this case the droplet $\setS$ is disjoint from $\supp\mu$, and so the "local" smoothness condition holds. (We will need this observation later.)

\subsection{Variational approach} \label{var} Here we sketch a different, more "physical'' proof of  our main result, Th \ref{mthm}. The proof is based on a variational argument well known in the physical literature, see e.g. the papers of Wiegmann and Zabrodin.
In the rigorous mathematical setting, this method was developed by
Johansson in the one-dimensional case, see \cite{J}.

We will use
the fact that the estimate \eqref{popp} for
$K_{m,n}(z,z)\e^{-mQ(z)}$ is uniform when we make small {\it smooth}
perturbations of the potential $Q$. We will also need some basic facts
concerning the variation of the droplet under the change of potential (Hele--Shaw theory).  Modulo  these technical issues (see Remark \ref{end})  the proof of the theorem is rather  short.

To simplify the notation we assume  $m=n$ and $\tau=1$
and write $K_n$ instead of $K_{n,n}$, etc.
Let  $h:\C\to\R$  be a bounded smooth function. We denote, for a positive integer $n$,
\begin{equation*}Q_{n}(z)=Q(z)-\frac{h(z)}n,\end{equation*}
and we will use "tilde-notation'' for various objects defined w.r.t.
the weight $Q_n$. Thus $\wt{K}_{n}$ is the kernel function w.r.t.
$Q_{n}$ etc., while the usual notation ($K_n$, etc.) is reserved for
the weight $Q$.

It is known that, for any $K\Subset\setS_1^\circ\cap X$, the
coincidence set
$\{Q_n=\lpar\widehat{Q}_n\rpar_1\}$, and therefore the perturbed droplet, will contain
$K$ in its interior when $n$ is large enough. One can then prove that
\begin{equation}\label{skum}\wt{K}_{n}(z,z)\e^{-nQ_n(z)}=n\Delta Q_n(z)+\frac 1 2
\Delta\log\Delta Q_n(z)+o(1),\qquad (n\to\infty),\end{equation} for $z\in K$, and that
the $o(1)$-term is uniform in  $z$.

Let $g\in\coity\lpar\setS_1^\circ\cap X\rpar$, so we have
\begin{equation*}\wt{K}_{n}(z,z)\e^{-nQ_n(z)}=n\Delta Q(z)- \Delta
h(z)+\frac 1 2 \Delta\log\Delta Q(z)+o(1)\end{equation*}
 uniformly for $z\in \supp g$.  We define
\begin{equation*}D_{n}^h[g]=\wt{E}_{n}\lpar \fluct_n
g\rpar.\end{equation*} If $V$ denotes the Vandermonde
determinant, we then have (see \eqref{zmno} and \eqref{vander})
\begin{equation*}\begin{split}D_{n}^{h}[g]&=\frac {\int_{\C^n}\fluct_n g\cdot
\babs{V}^2\e^{-n\trace_n Q_n}\dA_n} {\int_{\C^n}
\babs{V}^2 \e^{-n\trace_n Q_n}\dA_n}=\\
&=\frac {\int_{\C^n}\fluct_n g\cdot \e^{\trace_n
h}\babs{V}^2\e^{-n\trace_n Q}\dA_n} {\int_{\C^n} \e^{\trace_n
h}\babs{V}^2\e^{-n\trace_n Q}\dA_n} =\frac {E_{n}\lpar \fluct_n
g\cdot \e^{\trace_n h}\rpar} {E_n\lpar
\e^{\trace_n h}\rpar}.\\
\end{split}\end{equation*}
We now fix a real-valued $g$ and set
\begin{equation}\label{babb}h=\lambda \lpar g-\int g\Delta
Q\dA\rpar,\end{equation} where $\lambda$ is a real number, so that $$\trace_n h=\lambda \fluct_n g.$$
We have
\begin{equation*}D_{n}^{h}[g]=\frac {E_{n}\lpar \fluct_n g\cdot \e^{\lambda\fluct_n g}\rpar} {E_n\lpar
\e^{\lambda\fluct_n g}\rpar}=F_n^\prime(\lambda),\qquad
\text{where}\qquad F_{n}(\lambda):=\log \lpar
E_{n}\e^{\lambda\fluct_n g}\rpar.\end{equation*} Now from
\eqref{skum} we see that
\begin{equation*}\begin{split}D_{n}^{h}[g]&=\int_\C
g(z)\wt{K}_{n}(z,z)\e^{-nq_n(z)}\dA(z)-n\int_\C g\Delta Q\dA=\\
&=-\int\Delta h\cdot g\dA+\int
g\,\diff\nu+o(1)\to \int \d h \cdot\dbar g\dA+\int g\,\diff\nu.\\
\end{split}\end{equation*}
It follows from \eqref{babb} that
\begin{equation*}F_{n}^\prime(\lambda)\to \int
g\diff\nu+\frac \lambda {4}\int\babs{\nabla g}^2\dA\qquad
\text{as}\quad n\to\infty.\end{equation*} The last relation can
be integrated over $\lambda\in[0,1]$. This is justified by dominated
convergence and the estimate $F_{n}^{\prime\prime}\ge 0$, which is
just the Cauchy--Schwarz inequality. It follows that
\begin{equation*}\log E_{n}\e^{\fluct_n g}=F_{n}(1)=\int_0^1
F_{n}^\prime(\lambda)\diff\lambda\to \int g\diff\nu +\frac 1
8\int\babs{\nabla g}^2\dA\end{equation*} when $n\to\infty$. This
means that $$\log E_n\e^{t\fluct_n g}\to te_g+t^2v_g^2/2$$ for all
suitable scalars $t$, which in turn implies Th. \ref{mthm}.

\begin{rem} \label{end} We have discussed two rather different proofs for our main result Th. \ref{mthm}. We remark that, in the (most interesting) case when the potential is real-analytic in a neighbourhood of the droplet, the theory of asymptotic expansions for the correlation kernel is somewhat simpler and cleaner than in the smooth case. In the variational proof we need to make a smooth perturbation of the potential, and so we need a discussion of the smooth theory even in cases when the potential is real analytic. One would also need to include a further discussion of Hele-Shaw theory to make the variational proof
complete. We will discuss the variational approach in greater detail in our forthcoming paper \cite{AHM2}.
\end{rem}

\subsection{Interpretation in terms of Gaussian fields}\label{interp}
Denote $U=\setS_1^\circ\cap X$ and let $\calW_0(U)=W^{1,2}_0(U)$ be the
completion of $\coity(U)$ under the Dirichlet inner product
\begin{equation*}\langle
f,g\rangle_\nabla=\int_\C \nabla f\cdot\overline{\nabla g}~\dA.\end{equation*} Let
$G$ be the Green's function for $U$ and denote by
$\calE(U)=W^{-1,2}(U)$ the Hilbert space of
distributions  with inner product
\begin{equation*}\langle \rho_1,\rho_2\rangle_\calE=\int_U\int_U
G(z,w)~\diff\rho_1(z)~\diff\bar\rho_2(w).\end{equation*}
(More accurately, $\calE(U)$ is the completion of the space of measures with finite $\calE$-norm.) We have an
isomorphism
\begin{equation*}\Delta_U:\calW_0(U)\to\calE(U),\end{equation*} where
$\Delta_U=\d\dbar$ is the (Dirichlet) Laplacian. The inverse map is
given by the Green potential
\begin{equation*}-\frac 1 2 \Delta_U^{-1}\rho=U_G^\rho\qquad
\text{where}\qquad U_G^\rho(z)=\int_U
G(z,w)\diff\rho(w).\end{equation*} By a \textit{Gaussian field} indexed by
$\calW_0(U)$ we mean an isometry
\begin{equation*}\Gamma:\calW_0(U)\to L^2(\Omega,P),\end{equation*}
where $(\Omega,P)$ is some probability space, and $\Gamma(g)\sim
N\lpar 0,\|g\|_\nabla^2\rpar$ for any $g\in\calW_0(U)$. We now pick
$(\lambda_j)_1^n$ randomly w.r.t. $\Pi_{n,n}$ and consider the
sequence of random fields (measures)
\begin{equation*}\Gamma_n=4\lpar\sum_{j=1}^n
\delta_{\lambda_j}-n\sigma_1-\nu\rpar,\end{equation*} which satisfy
$$\Gamma_n(g)=4\lpar \fluct_n g-\int g\diff\nu\rpar.$$ Thus Th.
\ref{mthm} implies that  as $n\to\infty$, the fields $\Gamma_n$ converge to a Gaussian
field $\Gamma$  indexed by $\calW_0(U)$. The precise
meaning of the field convergence is convergence of the correlation functions:
\begin{equation}\label{wick}E_n\lpar \Gamma_n(g_1)\cdots \Gamma_n(g_k)\rpar\to
\langle \Gamma(g_1)\cdots\Gamma(g_k)\rangle\end{equation} for all
finite collections of test functions $\{g_j\}\subset\coity(U)$.
The right hand side in \eqref{wick} is given by the Wick's
formulas
\begin{equation*}\langle\Gamma(g_1)\cdots\Gamma(g_{2p+1})\rangle=0\end{equation*}
and
\begin{equation*}\langle\Gamma(g_1)\cdots\Gamma(g_{2p})\rangle=\sum\prod_{k=1}^p\langle
g_{i_k},g_{j_k}\rangle_\nabla,\end{equation*} where the sum is over
all partitions of $\{1,\ldots,2p\}$ into $p$ disjoint pairs
$(i_k,j_k)$.

Using the identifications mentioned above, we obtain the following
result.

\begin{prop} The random functions
\begin{equation*}h_n(z)=2\lpar\sum_{j=1}^n
G(z,\lambda_j)-U_G^{n\sigma_1+\nu}(z)\rpar,\end{equation*} converge in $U$
to a  Gaussian free field  with Dirichlet boundary condition, i.e.  to a Gaussian field indexed by $\calE(U)$.
\end{prop}

Alternatively, if we pick $(\lambda_j)$ and $(\lambda_j^\prime)$
independently w.r.t. $\Pi_{n,n}$ then the random functions
\begin{equation*}\tilde{h}_n(z)=\sum_{j=1}^n \lpar
G(z,\lambda_j)-G(z ,\lambda_j^\prime)\rpar\end{equation*} converge to a  Gaussian free field with Dirichlet boundary condition.

\subsection{Fluctuations near the boundary}\label{flubou}

In a separate publication \cite{AHM2} we will prove a version of Th. \ref{mthm}
valid for general test functions, which are not necessarily
supported in the droplet but just, say, of class
$\calC_0^\infty(\C)$. The proof is based on Ward's identities and Johansson's variational technique  mentioned above. Here we only settle for stating the result.

We assume throughout that $Q$ is real-analytic and strictly subharmonic in some neighbourhood
of the droplet $\calS=\calS_1$. One can then prove that the boundary $\d\calS$ is
\textit{regular}, i.e., a finite union of real-analytic curves. We
will write $\diff s$ for the arclength measure on $\d\setS_1$
divided by $2\pi$. Denote
\begin{equation*}U=\setS^\circ\qquad \text{and}\qquad
U_*=\C\setminus\setS.\end{equation*} We then have an orthogonal decomposition of the Sobolev space  $\calW=W^{1,2}(\C)$,
\begin{equation*}\calW=\calW_0(U)\oplus \calW\lpar \d
\calS\rpar\oplus\calW_0(U_*).\end{equation*} Here $\calW_0(U)$ and
$\calW_0(U_*)$ are identified with the subspaces of functions which
are (quasi-everywhere) zero in the complement of $U$ and $U_*$ respectively, while
the subspace $\calW\lpar \d\calS\rpar$ consists of the functions
which are harmonic off  $\d\calS$.
The orthogonal projection of $\calW$ onto $\calW(\d\setS)$,
\begin{equation*}f\mapsto f^{\d\calS},\end{equation*}
is just the composition of the restriction operator $f\mapsto
f\big|_{\d\calS}$ and the operation of harmonic extension to
$U\cup U_*\cup\{\infty\}$. For $f\in\calW$ we also denote by
$f^{\calS}$ the orthogonal projection of $f$ onto
$\calW_0(U)\oplus\calW\lpar \d\calS\rpar$,
\begin{equation*}f^{\calS}={\bf 1}_{\calS}\cdot f+{\bf
1}_{U_*}\cdot f^{\d\calS},\end{equation*} in other words, $f^{\calS}$
coincides with $f$ on $\calS$ and is harmonic and bounded in the
complement of that set.

 Finally, we write $n_U f$ for
the exterior normal derivative of $f\big|_{\setS}$ and $n_{U_*}f$
the exterior normal derivative of
$f^{\d\setS}\big|_{U_*}$.
We can now state the theorem.

\begin{thm} \label{grth} Let $f\in\coity(\C)$. Then the random variables $\fluct_n f$ on the
space $(\C^n,\Pi_{n,n})$ converge in distribution
to $N\lpar e_f,v_f^2\rpar$, where
\begin{equation*}v_f^2=\frac 1 4\int\babs{\nabla
\lpar f^{\calS}\rpar}^2\dA,\end{equation*} and
\begin{equation*}e_f=\int_{\calS}f\diff\nu+\frac 1 {4}
\int_{\d\calS}n_U(f)\diff s+\frac 1 {4}\int_{\d\calS}\lpar
f\cdot n_{U_*}\lpar \log\Delta Q\rpar- n_{U_*}\lpar
f^{\d\calS}\rpar\cdot\log\Delta Q\rpar\diff s.\end{equation*}
\end{thm}

Note that the formula for $e_f$ becomes very simple in the case of the so called
\textit{Hele--Shaw} potentials , i.e. if $\Delta
Q=\text{const.}>0$ in a neighbourhood of $\setS$, then
\begin{equation}e_f=\frac 1 4\int_{\d\setS}n_U(f)\diff
s.\end{equation}
In field theoretical terms, Th. \ref{grth} means that the random
measures
\begin{equation*}4\lpar
\sum_{j=1}^n\delta_{\lambda_j}-n\sigma_1-\nu\rpar\end{equation*}
converge in $\C$ to the sum of two independent Gaussian fields --  indexed by
$\calW_0(U)$ and by $\calW\lpar\d\setS_1\rpar$ respectively.
 While the first one is conformally invariant, the
second one is not.

Alternatively, we can say that the random functions
\begin{equation*}h_n(z)=\log\babs{\frac {p(z;M_1)}
{p(z;M_2)}},\end{equation*} where the $p(z;M_j)$ are the
characteristic polynomials of two independent $n\times n$ random
normal matrices $M_j$, converge to a free Gaussian field on $\setS$
with {\it free} boundary condition.

\subsection{Large volume limit} \label{ginibre}
Let us take a point $z_0\in \setS_1^\circ\cap X$ and assume for simplicity that
$\Delta Q(z_0)=1$. Define   $\mu_n\in$ Prob$(C^n)$ as the image of
$\Pi_{n,n}$ under the map $$(\lambda_j)_{j=1}^n\mapsto \lpar
\sqrt{n}(\lambda_j-z_0)\rpar_{j=1}^n,$$ and think of $\mu_n$ as a point process in $\C$.

\begin{prop}   The processes $\mu_n$ converge to the $\text{Ginibre}(\infty)$ point process, i.e. to the
determinantal process with correlation kernel
$$K(z,w)=\e^{z\bar{w}-(\babs{z}^2+\babs{w}^2)/2}.$$\end{prop}

\begin{proof}
Assume w.l.o.g.  $z_0=0$. Then $\mu_n$ are
determinantal processes with correlation kernels
\begin{equation*}k_n(z,w)=\frac 1 n K_{n,n}\lpar \frac z
{\sqrt{n}},\frac w {\sqrt{n}}\rpar\e^{-n\lpar Q( z/\sqrt{n})+Q(w/
{\sqrt{n}})\rpar/2}.\end{equation*} Using the expansion for
$K_{n,n}$ in Lemma \ref{klam}, we see that
\begin{equation*}k_n(z,w)=\lpar \Delta
Q(0)+o(1)\rpar\e^{n\psi\lpar z/\sqrt{n},\bar{w}/\sqrt{n}\rpar-n\lpar
Q(z/\sqrt{n})+Q(w/\sqrt{n})\rpar/2},
\end{equation*}
where the $o(1)$ is uniform for $z$ and $w$ in a fixed compact
subset of $\C$. Next observe that, up to negligible terms, we have
\begin{equation*}\psi\lpar
z,\bar{w}\rpar=Q(0)+az+\bar{a}\bar{w}+bz^2+\bar{b}\bar{w}^2+z\bar{w},\end{equation*}
for some complex numbers $a$ and $b$. It follows that
\begin{equation*}k_n(z,w)=\lpar 1+o(1)\rpar
\e^{\imag\sqrt{n}\im\lpar a(z-w)\rpar} \e^{\imag\im\lpar
b(z^2-w^2)\rpar}\e^{z\bar{w}-\lpar
\babs{z}^2+\babs{w}^2\rpar/2}.\end{equation*} The first two
exponential factors cancel out when we compute the determinants
representing intensity $k$-point functions, which yields the desired
result.
\end{proof}

\subsection{Berezin transform and fluctuations of eigenvalues}\label{skiit}
 We will write
\begin{equation*}R_n^k(\lambda_1,\ldots,\lambda_k)=\det\lpar
K_n(\lambda_i,\lambda_j)\rpar_{i,j=1}^k\e^{-n\sum_{j=1}^k
Q(\lambda_j)}\end{equation*} for the  $k$-point {\it intensity
function} of the ensemble \eqref{vander} with $m=n$. We will also need the \textit{connected} $2$-point function
\begin{equation*}R_n^{2,c}(z,w)=R_n^2(z,w)-R_n^1(z)R_n^1(w)=-\babs{K_n(z,w)}^2\e^{-n(Q(z)+Q(w))}.
\end{equation*}
 It is
 easy to check that
\begin{equation*}\int_\C
R_n^{2,c}(z,w)\dA(w)=-R_n^1(z),\end{equation*} and
\begin{equation*}\Cov\lpar \fluct_n f,\fluct_n g\rpar=\int_\C
f(z)g(z)R_n^1(z)\dA(z)+\int_{\C^2}f(z)g(w)R_n^{2,c}(z,w)\dA_2(z,w).\end{equation*}
Recall that for a given $z$,  the corresponding {\it Berezin kernel}
$\berd^{\langle z\rangle}_n$ is given
by
\begin{equation*}\berd^{\langle z\rangle}_n(w)=-\frac
{R_n^{2,c}(z,w)} {R_n^1(z)}=R_n^1(w)-\frac {R_n^2(z,w)}
{R_n^1(z)},\end{equation*} and the {\it Berezin transform} is
\begin{equation*}\calB_n f(z)=\int_\C f(w)\berd^{\langle
z\rangle}_n(w)\dA(w).\end{equation*}
We may now conclude that
\begin{equation*}\Cov\lpar\fluct_n f,\fluct_n g\rpar=\int_\C\lpar
f(z)-\calB_n f(z)\rpar g(z) R_n^1(z)\dA(z).\end{equation*} On the
other hand, Th. \ref{mthm} implies that
\begin{equation*}\Cov_n\lpar \fluct_n f,\fluct_n g\rpar\to-\int_\C
\Delta f(z)g(z)\dA(z),\qquad(n\to\infty),\end{equation*}  where
$f,g\in\coity\lpar\calS_1^\circ\cap X\rpar$. Therefore,
\begin{equation*}\int \lpar f(z)-\calB_n f(z)\rpar R_n^1(z)g(z)\dA(z)\to -\int\Delta
f(z)g(z)\dA(z).\end{equation*} Since
$$R_n^1=n\Delta Q+\frac 1 2 \Delta\log\Delta Q(z)+o(1)$$ on the support of
$g$, we  obtain
 the following asymptotic formula for the Berezin transform.

\begin{prop} If $f\in\coity\lpar \setS_1^\circ\cap X\rpar$, then
\begin{equation}\calB_n f= f+\frac {\Delta
f} {n\Delta Q}+o\left(\frac1n\right)\end{equation}
inside the droplet in the
sense of distributions.\end{prop}

Berezin's transform has the following {\it probabilistic interpretation}. Let us think of the measure $\Pi_n=\Pi_{n,n}$
as the law of a point process  $\Phi_n$ in $\C$. We will refer to $\Phi_n$
as the $n$-point {\it RNM} (random normal matrix)
{\it process}   associated with
potential $Q$.

Let us now condition   $\Phi_n$ on the event  $\{z_0\in\Phi_n\}$ and  write $\wt{\Phi}^{\langle z_0\rangle}_{n-1}$ for conditional
$(n-1)$-point process.
 Accordingly, we write $R_n^k$ for the $k$-point intensity
function of $\Phi_n$ and $\wt{R}_{n-1}^k=\wt{R}_{n-1}^{k,\langle
z_0\rangle}$ the $k$-point function of $\wt{\Phi}_{n-1}^{\langle
z_0\rangle}$.

\begin{lem}
\begin{equation}\label{calm}\berd^{\langle
z_0\rangle}_n(z)=R_n^1(z)-\wt{R}_{n-1}^1(z).\end{equation} \end{lem}
\begin{proof} Consider small discs $D$ and $D_0$ centered at $z$ and $z_0$ with radii $\eps$ and $\eps_0$ respectively. We have
\begin{equation*}R_n^1(z_0)=\lim_{\eps_0\to 0}\frac {\Pi_n\lpar\big\{
\Phi_n\cap D_0\ne \emptyset\big\}\rpar}
{\eps_0^2},\end{equation*} and
\begin{equation*}R_n^2(z_0,z)=\lim_{\eps,\eps_0\to 0}\frac
{\Pi_n\lpar \big\{\Phi_n\cap D\ne\emptyset\big\}\cap\big\{
\Phi_n\cap D_0\ne \emptyset\big\}\rpar}
{\eps^2\eps_0^2}.\end{equation*} It follows  that
\begin{equation*}\begin{split}\wt{R}_{n-1}^1(z)&=\lim_{\eps\to
0}\lim_{\eps_0\to 0}\frac {\Pi_n\lpar\big\{\Phi_n\cap
D\ne\emptyset\, \big|\, \Phi_n\cap D_0\ne\emptyset\big\}\rpar}
{\eps^2}=\\
&=\lim_{\eps\to 0}\lim_{\eps_0\to 0}\frac {\Pi_n\lpar
\big\{\Phi_n\cap D\ne\emptyset\big\}\cap\big\{ \Phi_n\cap D_0\ne
\emptyset\big\}\rpar} {\eps^2\Pi_n\lpar\big\{ \Phi_n\cap D_0\ne
\emptyset\big\}\rpar}=\\
&=\frac {R_n^2(z_0,z)} {R_n^1(z_0)}=R_n^1(z)-\berd^{\langle
z_0\rangle}_n(z).\\
\end{split}\end{equation*}
\end{proof}

Integrating \eqref{calm}
 against test functions we get
the following formula, where $E_n$ stands for the
expectation with respect to $\Pi_n$ and $\wt{E}_{n-1}^{\langle
z_0\rangle}$  with respect to the law of
$\wt{\Phi}_{n-1}^{\langle z_0\rangle}$.

\begin{cor} \label{BEF}Let $z_0\in\C$ and $f\in \calC_b(\C)$. Then
\begin{equation*}\calB_n f(z_0)=E_n\lpar \trace_n
f\rpar-\wt{E}_{n-1}^{\langle z_0\rangle}\lpar\trace_{n-1} f\rpar.
\end{equation*}
\end{cor}

The {\it central limit theorem} for  Berezin transform
 states that the rescaled (as in the large volume limit procedure)  Berezin's measures
converge to the standard
Gaussian distribution in $\C$, see \cite{AH},Th. 2.6.
We can now interpret this statement in terms of random eigenvalues.

Let $z_0\in\setS_1^\circ\cap X$ and  assume w.l.o.g. that $\Delta
Q(z_0)=1$. Define $\wh{\Phi}_{n-1}^{\langle z_0\rangle}$ as a point process in $\C$ obtained from
$\wt{\Phi}^{\langle z_0\rangle}_{n-1}$  by dilating all distances to $z_0$ by a factor of
$\sqrt{n}$ as in the previous subsection. In other words, we condition $\Phi_n$ on the event "$z_0$ is an eigenvalue'' and rescale the distances.

\begin{prop} The limiting point process
of  $\wh{\Phi}_n^{\langle z_0\rangle}$, ($n\to\infty$),   has the following  one-point intensity function:
$$\wh{R}^{1,\langle z_0\rangle}(z)=1-\e^{-\babs{z-z_0}^2}.$$
\end{prop}

\begin{proof}
Let $\wh{R}_{n-1}^{1,\langle z_0\rangle}$ denote the
one-point function of $\wh{\Phi}_{n-1}^{\langle z_0\rangle}$.
Similarly, let $\wh{R}_n^1$ be the one-point
function for the process $\wh{\Phi}_n$, by which we mean $\Phi_n$
dilated by a factor of $\sqrt{n}$ about $z_0$. By Proposition~8.3, the point processes
$\wh{\Phi}_n$ converge to   Ginibre$(\infty)$ ensemble as $n\to\infty$. The
one-point function of  Ginibre$(\infty)$ is $\wh{R}^1(z)\equiv1$ and its Berezin kernel is
$\wh{\berd}^{\langle z_0\rangle}(z)=\e^{-\babs{z-z_0}^2}$. Conditioning
the equation \eqref{calm} on the event "$z_0$ is an eigenvalue'', we get
\begin{equation*}\wh{\berd}_n^{\langle
z_0\rangle}(z)=\wh{R}_n^1(z)-\wh{R}_{n-1}^{1,\langle
z_0\rangle}(z),\end{equation*} and sending $n\to\infty$ we get the
stated formula.
\end{proof}

\subsection{Berezin transform in quasi-classical limit and orthogonal polynomials}
 As before, let  $\Phi_n$ be the $n$-point RNM process associated with
potential $Q$. We fix a point $z_0$ and condition $\Phi_n$ on the event $\{z_0\in\Phi_n\}$.

\begin{lem} \label{yxa} The conditional $(n-1)$-point process $\wt{\Phi}_{n-1}^{\langle z_0\rangle}$ is the RNM process associated with the potential
\begin{equation*}\wt{Q}(z)=Q(z)-\frac1{n-1}\lpar\log\babs{z-z_0}^2-Q(z)\rpar.\end{equation*}
\end{lem}

\begin{proof} The density of the measure $\Pi_n$ is given by
\begin{equation}\label{ukno}\rho(\lambda_1,\ldots,\lambda_n)=\frac 1 {Z}\babs{V_n(\lambda_1,\ldots,\lambda_n)}^2\e^{-n\lpar Q(\lambda_1)+\ldots+Q(\lambda_n)\rpar},\end{equation}
where $Z$ is the normalizing factor (partition function) and $V_n$ the Vandermonde determinant, see \eqref{zmno}.
Setting $z_0=\lambda_n$, we have
\begin{equation}\begin{split}\rho(\lambda_1,\ldots,\lambda_{n-1},z_0)&=\frac {e^{-nQ(z_0)}} {Z}\babs{V_{n-1}(\lambda_1,\ldots,\lambda_{n-1})}^2\e^{-n\lpar Q(\lambda_1)+\ldots+Q(\lambda_{n-1})\rpar
+\sum_{j=1}^{n-1}\log\babs{\lambda_j-z_0}^2}\\
&=\frac {e^{-nQ(z_0)}} { Z}\babs{V_{n-1}(\lambda_1,\ldots,\lambda_{n-1})}^2\e^{-(n-1)\lpar \wt{Q}_n(\lambda_1)+\ldots+\wt{Q}(\lambda_{n-1})\rpar}.\\
\end{split}
\end{equation}
It follows that the density of the conditional point process  $\wt{\Phi}_{n-1}^{\langle z_0\rangle}$ is
$$\tilde \rho(\lambda_1,\ldots,\lambda_{n-1})=\frac 1 {\wt{Z}}\babs{V_{n-1}(\lambda_1,\ldots,\lambda_{n-1})}^2\e^{-(n-1)\lpar \wt{Q}_n(\lambda_1)+\ldots+\wt{Q}_n(\lambda_{n-1})\rpar},$$
where $\tilde Z$ is the corresponding normalizing factor.
\end{proof}

\smallskip
 Let us now assume that the potential $Q$ is real analytic and   strictly subharmonic
in some neighbourhood of the droplet $\setS=\setS_1$ so that Theorem \ref{grth} applies. Denote
\begin{equation*}\wt{Q}_n(z)=Q(z)-\frac{h(z)}n,\qquad h(z):=\log\babs{z-z_0}^2-Q(z).\end{equation*}
i.e. so that $\wt{Q}_n=Q-h/n$. As in Subsection~\ref{var}, for a bounded smooth function $f$ we write
$$D_n[f]=E_n\lpar \fluct_n f\rpar, \qquad
D_n^h[f]=\wt{E}_n\lpar\fluct_n f\rpar,$$
 where $\wt{E}_n$ is the expectation with respect to the potential $\wt{Q}_n$.

The argument  in Section~8.2 shows  that the variance part of Theorem \ref{grth} is equivalent to the statement that
\begin{equation*}D_n[f]-D_n^h[f]\to \frac 1 4\langle f^{\setS},h\rangle_{\nabla},\end{equation*}
where $f^{\setS}$ is the orthogonal projection of $f$ onto $\calW_0(U)\oplus \calW(\d\setS)$. By Corollary~\ref{BEF} and Lemma~\ref{yxa}, we
have
\begin{equation*}\begin{split}\calB_n f(z_0)&=E_n\lpar \trace_n f\rpar-\wt{E}_{n-1}\lpar\trace_{n-1} f\rpar\\&=\int f\diff\sigma+
E_n\lpar \fluct_n f\rpar-E_{n-1}\lpar\fluct_{n-1} f\rpar\\
&=\int f\diff\sigma+D_n[f]-D_{n-1}^h[f],
\end{split}\end{equation*}
and therefore
\begin{equation}\label{scout}\calB_n f(z_0)\to\int f^{\setS}\diff\sigma+\langle f^{\setS},h\rangle_\nabla,\qquad (n\to\infty).\end{equation}
Note that
\begin{equation*}\langle f^{\setS},h\rangle_\nabla=\langle f^{\setS},Q^{\setS}\rangle_\nabla-\langle f^{\setS},l\rangle_\nabla,\end{equation*}
where $l(z)=\log\babs{z-z_0}^2$ and
\begin{equation*}\langle f^{\setS},Q^{\setS}_n\rangle_\nabla=-\int f^{\setS}\Delta Q^{\setS}\dA=-\int f\diff\sigma,\end{equation*}
and
\begin{equation*}-\langle f^{\setS},l\rangle_\nabla=\int f^{\setS}\Delta l\dA=f^{\setS}(z_0).\end{equation*}
In view of \eqref{scout}, it follows that
\begin{equation*}\calB_n f(z_0)\to f^{\setS}(z_0).\end{equation*}
Since the function $f$ was arbitrary, we have derived  the following result.

\begin{thm} \label{yto} Let $z_0\in\C$. Then
the Berezin measures $B_n^{\langle z_0\rangle}\dA$ converge to the Dirac measure at $z_0$ if $z_0\in \setS_1$, and to the
harmonic measure of  $\C\setminus \setS_1$ evaluated at $z_0$ if $z_0\not\in\setS_1$.
\end{thm}

This theorem is also true at $z_0=\infty$, in which case it has the following form.

\begin{thm} \label{yto2} Let $P_n$ be the $n$-th orthonormal polynomial with respect to the  measure $\e^{-nQ}\dA$ in $\C$. Then the probability measures $$|P_n|^2\e^{-nQ}\dA$$ converge
to the harmonic measure of  $\hat \C\setminus\setS_1$ evaluated at $\infty$.
\end{thm}

\begin{proof} We need to compute the limit of the Berezin kernel $\berd_n^{\langle z_0\rangle}(z)$ as $z_0\to\infty$. By Lemma~\ref{calm} we have
$$\berd_n^{\langle z_0\rangle}(z)=R^1_n(z)-\tilde R^1_{n-1}(z),$$
where $R^1_n$ and $\tilde R^1_{n-1}$ are the 1-point functions of  $\Phi_n$ and $\wt{\Phi}_{n-1}^{\langle z_0\rangle}$ respectively. Since  $\Phi_n$ is the $n$-point RNM process associated with potential $Q$, we have
$$R^1_n=\sum_{k=0}^{n-1}|P_k|^2~e^{-nQ}.$$
On the other hand, by Lemma~\ref{yxa},  $\wt{\Phi}_{n-1}^{\langle z_0\rangle}$ is the $(n-1)$-point RNM process associated with the potential
$$\tilde Q^{<z_0>}(z)=\frac n{n-1}Q(z)+\frac1{n-1}\log\lpar \frac {\babs{z_0}^2} {\babs{z-z_0}^2}\rpar.$$
(Here we added a constant term to the potential $\tilde Q$ in Lemma~\ref{yxa}; this clearly didn't affect the point process.)
Since
$$\tilde Q^{<z_0>}(z)\to \tilde Q(z):=\frac n{n-1}Q(z)\quad{\rm as}\quad z_0\to\infty,$$
we have
$$\lim_{z_0\to\infty} \tilde R^1_{n-1}=\sum_{k=0}^{n-2}|\tilde P_k|^2~e^{-(n-1)\tilde Q},$$
where $\{\tilde P_k\}$ are orthonormal polynomials with respect to the weight
$$e^{-(n-1)\tilde Q}=e^{-nQ}.$$
Since the weight is the same for the polynomials $\{P_k\}$ and $\{\tilde P_k\}$ we have
$$\berd_n^{\langle \infty\rangle}=\sum_{k=0}^{n-1}|P_k|^2~e^{-nQ}-\sum_{k=0}^{n-2}|P_k|^2~e^{-nQ}=|P_{n-1}|^2~e^{-nQ}.$$
Combining this with Theorem~\ref{yto} we conclude the proof.
 \end{proof}

\subsection{Further remarks on the cumulant method}\label{compare}
We here continue our discussion of the cumulant method (Sect \ref{main1})
and compare our result with some other
related work using this method.

In \cite{S1}, Soshnikov studied
linear statistics of the form
$\trace_n g_n-E(\trace_n g_n)$ where
$g_n(t)=g(L_n t)$ and $L_n$ is a fixed sequence with $L_n\to\infty$, $L_n/n\to 0$. The expectation is here understood with respect
to the classical Weyl measure on $[-\pi,\pi)^n$, i.e., we
are considering the Gaussian unitary ensemble; $g:\R\to\R$ is a test function in the Schwarz class.

 In \cite{S1}, asymptotic normality is proved for these linear statistics using the cumulant method applied to the sine-kernel,
i.e. the explicit correlation kernel in that case. The asymptotic variance of $\trace_n g_n$ turns out to be finite and independent
of the particular sequence $L_n$; it equals
$\frac 1 {2\pi}\int_\R \babs{\hat{g}(t)}^2\babs{t}~\diff t$.

The method in \cite{S1} does however not allow to draw conclusions about the case $L_n\approx 1$;
the assumption $L_n\to\infty$ is used in the proof of Theorem 1 (p. 1357), where limits of certain
Riemann sums are identified.

We also want to mention the short proof of asymptotic normality due to
Costin and Lebowitz \cite{CL}. In the situation of \cite{CL}, one considers certain linear statistics which have infinite asymptotic variance. This infiniteness of the
variance is then used to show decay of the cumulants of the corresponding normalized variables.
(Thus the method in \cite{CL} necessarily breaks down in our situation,
when the variance tends to a finite limit.)

The cumulant method has also been used in the theory of Gaussian analytic functions, see \cite{ST}. In this case, asymptotic normality was
obtained for linear statistics
 whose variances converge to zero. In \cite{SZ}, the result was generalized to a setting of zeros of random holomorphic sections of high powers of a positive Hermitian line bundle over a K\"{a}hler manifold.
 (Cf. the book \cite{HKPV} for further developments in the theory of Gaussian analytic functions.)

\bigskip

\subsubsection*{Acknowledgements.} We are grateful to Alexei
Borodin, Kurt Johansson and Paul Wiegmann for help and useful
discussions.

\end{document}